\documentclass[10pt, a4paper]{article}

\usepackage{amsmath, amssymb, amsthm}
\usepackage{vmargin}
\usepackage[UKenglish]{babel}

\newcommand{\<}{\langle}

\newcommand{\C}{\mathbb{C}}
\newcommand{\K}{\mathbb{K}}
\newcommand{\Z}{\mathbb{Z}}

\DeclareMathOperator{\rk}{rk}
\DeclareMathOperator{\cork}{cork}

\DeclareMathOperator{\SL}{SL}
\DeclareMathOperator{\PSL}{PSL}
\DeclareMathOperator{\PGL}{PGL}
\DeclareMathOperator{\pSL}{(P)SL}
\DeclareMathOperator{\GL}{GL}

\DeclareMathOperator{\End}{End}
\DeclareMathOperator{\Id}{Id}

\newtheorem*{theorem*}{Theorem}
\newtheorem*{proposition}{Proposition}
\newtheorem*{notation*}{Notation}
\newtheorem*{definition*}{Definition}
\newtheorem{lemma}{Lemma}
\newtheorem*{conjecture*}{Conjecture}

\newcommand{\theoremname}{}
\newtheorem*{generictheorem*}{\theoremname}

\theoremstyle{remark}
\newtheorem*{notationinproof*}{Notation}
\newtheorem{step}{Step}[subsubsection]

\renewcommand{\thestep}{\arabic{step}}

\newenvironment{proofclaim}{\begin{proof}[Proof of Step \thestep]}{\end{proof}}

\title{Rank $3$ Bingo}
\author{Alexandre Borovik and Adrien Deloro}

\begin{document}

\renewcommand{\o}{^{\circ}}
\renewcommand{\>}{\rangle}
\renewcommand{\L}{\mathbb{L}}

\maketitle

\begin{abstract}
{We classify irreducible actions of connected groups of finite Morley rank on abelian groups of Morley rank $3$.}
\end{abstract}

\section{The Result and its Context}

\subsection{The context}
The present article deals with representations of groups of finite Morley rank. Morley rank is the logician's coarse approach to Zariski dimension; a good general reference on the topic is \cite{BNGroups}, where the theory is systematically developed, and the reader not too familiar with the subject may start there. Following one's algebraic intuition is another possibility, as groups of finite Morley rank behave in many respects very much like algebraic groups over algebraically closed fields do. This intuition shaped the famous Cherlin-Zilber Conjecture: \emph{simple infinite groups of finite Morley ranks are simple algebraic groups over algebraically closed fields}.

However, since there is no rational structure around, the Cherlin-Zilber Conjecture is still an open question, and the present setting is broader than the theory of algebraic groups. On the other hand, finite groups are groups of Morley rank $0$, and it had happened that methods of classification of finite simple groups could be successfully applied to the general case of groups of finite Morley rank. This became, over the last 20 years, the principal line of development and resulted, in particular, in confirmation of the Cherlin-Zilber conjecture in a number of important cases, see \cite{ABCSimple}.

This paper (together with \cite{BBGroups,BCPermutation,CDSmall,DActions,WRank4}) signals a shift in the direction of research in the theory of groups of finite Morley rank: instead of the study of their internal structure  we focus on the study of \emph{actions} of groups of finite Morley rank.

Indeed groups of finite Morley rank naturally arise in model theory as a kind of Galois groups and have an action naturally attached to them. More precisely, any uncountably categorical structure is controlled by certain definable groups of permutations which have finite Morley rank, by definability. This observation leads to the concept of a \emph{binding group} \cite[\S2.5]{PStable}, introduced by Zilber and developed in other contexts by Hrushovski, an important special case being that of Lie groups in the Picard-Vessiot theory of linear differential equations.

But as another consequence of the absence of a rational structure, representations (permutation and linear) in the finite Morley rank category must be studied by elementary means. The topic being rather new we deal in this paper with a basic case: actions on a module of rank $3$, for which we provide a classification.

\subsection{The result}

\begin{theorem*}
Let $G$ be a connected, non-soluble group of finite Morley rank acting definably and faithfully on an abelian group $V$ of Morley rank $3$. Suppose that $V$ is $G$-minimal. Then:
\begin{itemize}
\item
either $G = \PSL_2(\K) \times Z(G)$ where $\PSL_2(\K)$ acts in its adjoint action on $V\simeq \K_+^3$,
\item
or $G = \SL_3(\K) \ast Z(G)$ in its natural action on $V \simeq \K_+^3$,
\item
or $G$ is a simple bad group of rank $3$, and $V$ has odd prime exponent.
\end{itemize}
\end{theorem*}

In the algebraic category, irreducible, three-dimensional representations are of course well-known. But this is the whole point: to prove that the pair $(G,V)$ lives in the algebraic group category. The principal  difficulties are related to the possibility of so-called \emph{bad groups}, on which we say more in the prerequisites.

Interestingly enough, our proof involves ideas from more or less all directions explored over almost forty years of groups of finite Morley rank. The present article is therefore the best opportunity we shall ever have to print our heartful thanks to all members of the ranked universe: Tuna, Christine, Oleg,
Ay\c{s}e, Jeffrey, Gregory, Luis-Jaime, Olivier, Ursula, Ehud,
the late \'Eric, James, Angus, Dugald, Yerulan,
Ali, Anand, Bruno, Katrin, Jules, P\i nar,
Frank, Joshua, and Boris
(with our apologies to whomever we forgot). The reader can play bingo with these names and match them against the various results we shall mention.

And of course, our special extra thanks to Ali, mayor of the Matematik Koy\"{u} at \c{S}irince, Turkey -- this is one more result proved there.

\subsection{Future directions}

The result of this paper deals with a configuration that arises in  bases of induction (on Morley rank) in proofs of more general results on representations in the finite Morley rank category. One of the examples is the following work-in-progress result by Berkman and the first author:

\begin{theorem*}[Berkman and Borovik, work in progress]
Let  $H$ and $V$ be connected groups of finite Morley rank and $V$  an elementary abelian $p$-group for $p\neq 2$ of Morley rank $n>2$. Assume that $H$ acts on $V$ definably, and the action is faithful and generically $n$-transitive.

Then there is an algebraically closed field $F$ such that $V\cong F^n$ and $H\cong\operatorname{GL}(V)$, and the action is the natural action.
\end{theorem*}

This theorem, in its turn, is needed for confirming a conjecture that improves bounds from \cite{BCPermutation} and makes them sharp.

\begin{conjecture*}
Let $G$ be a connected group of finite Morley rank acting faithfully, definably, transitively and generically $k$-transitively on a set $X$ of Morley rank $n$ {\rm (}that is, has an orbit on $X^{k}$ of the same Morley rank as $X^{k}${\rm )}. Then $k \leqslant n+2$, and if, in addition, $k=n+2$ then the pair $(G,X)$ is equivalent to the projective general linear group $\PGL_{n+1}(F)$ acting on the projective space ${\mathbb{P}}^n(F)$ for some algebraically closed field $F$.
\end{conjecture*}

(Actually, $V \rtimes H$ from the previous tentative result is likely to appear $G$ as the stabiliser of a generic point in $X$.)

The conjecture above is ideologically very important: it bounds the complexity of permutation groups of finite Morley rank exactly at the level of ``classical'' mathematics and canonical examples.

So perhaps it should not be surprising that the present paper that looks at one of the special configurations in the basis of induction uses the total of the research on groups of finite Morley rank accumulated over 40 years.

\subsection{Prerequisites}

The article is far from being self-contained as we assume familiarity with a number of topics: definable closure \cite[\S5.5]{BNGroups}, connected component \cite[\S5.2]{BNGroups}, torsion lifting \cite[ex.11 p.98]{BNGroups}, Zilber's Indecomposibility Theorem \cite[\S5.4]{BNGroups}, the structure of abelian and nilpotent groups \cite[\S6.2]{BNGroups}, the structure of soluble $p$-subgroups \cite[\S6.4]{BNGroups}, the Pr\"{u}fer $p$-rank $\Pr_p(\cdot)$, $p$-unipotent subgroups and the $U_p(\cdot)$ radical \cite[\S2.3]{DJInvolutive}, the Fitting subgroup $F\o(\cdot)$ \cite[\S7.2]{BNGroups}, Borel subgroups \cite[\S2.4]{DJInvolutive}, fields of finite Morley rank \cite[\S8.1]{BNGroups}, Sylow $2$-subgroups \cite[\S10.3]{BNGroups}, good tori \cite{CGood}, torality principles \cite[Corollary 3]{BCSemisimple}.
There are no specific prerequisites on permutation groups, but \cite{MPPrimitive} can provide useful background.
More subjects will be mentioned in due time; for the moment let us quote only the key results and methods.


Recall that a \emph{bad group} is a (potential) group of finite Morley rank all definable, connected, proper subgroups of which are nilpotent. Be careful that the condition is on \emph{all} proper subgroups, and that one does not require simplicity. Bad groups of rank $3$ were encountered by Cherlin in the very first article on groups of finite Morley rank \cite{CGroups}; we still do not know whether these do exist, but they have been extensively studied, in particular by Cherlin, Nesin, and Corredor.

\renewcommand{\theoremname}{Bad Group Analysis}
\begin{generictheorem*}[{from \cite[Theorem 13.3 and Proposition 13.4]{BNGroups}}]
Let $G$ be a simple bad group. Then the definable, connected, proper subgroups of $G$ are conjugate to each other, and $G$ has no involutions.
Actually $G$ has no definable, involutive automorphism.
\end{generictheorem*}

We now start talking about group actions. First recall two facts on semi-simplicity.

\renewcommand{\theoremname}{Wagner's Torus Theorem}
\begin{generictheorem*}[{\cite{WFields}}]
Let $\K$ be a field of finite Morley rank of positive characteristic. Then $\K^\times$ is a good torus.
\end{generictheorem*}

\renewcommand{\theoremname}{Semi-Simple Actions}
\begin{generictheorem*}[{\cite[Lemma G]{DJInvolutive}}]
In a universe of finite Morley rank, consider the following definable objects: a definable, soluble group $T$ with no elements of order $p$, a connected, elementary abelian $p$-group $A$, and an action of $T$ on $A$. Then $A = C_A(T) \oplus [A, T]$. Let $A_0 \leq A$ be a definable, connected, $T$-invariant subgroup. Then $C_A(T)$ covers $C_{A/A_0}(T)$ and $C_T(A) = C_T(A_0, A/A_0)$.
\end{generictheorem*}

This will be applied with $T$ a cyclic group or $T$ a good torus with no elements of order $p$. Parenthetically said, Tindzogho Ntsiri has obtained in his Ph.D. \cite[\S5.2]{TEtude} an analogue to Maschke's Theorem for subtori of $\K^\times$ in positive characteristic.
%
%

When the acting group is not a torus, much less is known -- whence the present article.
The basic case is the action on a strongly minimal set.

\renewcommand{\theoremname}{Hrushovski's Theorem}
\begin{generictheorem*}[{from \cite[Theorem 11.98]{BNGroups}}]
Let $G$ be a connected group of finite Morley rank acting definably, transitively, and faithfully on a set $X$ with $\rk X = \deg X = 1$. Then $\rk (G) \leq 3$, and if $G$ is non-soluble there is a definable field structure $\K$ such that $G \simeq \PSL_2(\K)$.
\end{generictheorem*}

Incidently, Wiscons pursued in this permutation-theoretic vein and could classify non-soluble groups of Morley rank $4$ acting sufficiently generically on sets of rank $2$ \cite[Corollary B]{WRank4}, extending and simplifying earlier work by Gropp \cite{GThere}. Although some aspects of Wiscons' work are extremely helpful in the proof below, most of our configurations will be more algebraic as we shall mainly act on modules.

One word on terminology may be in order. We reserve the phrase $G$-module for a definable, \emph{connected}, abelian group acted on by $G$. Accordingly, reducibility refers to the existence of a non-trivial, proper $G$-submodule $W$: definability and connectedness of $W$ are therefore required. Likewise, a $G$-composition series $0 = V_0 < \dots < V_\ell = V$ being a series of $G$-submodules of maximal length $\ell_G(V) = \ell$, the $V_i$'s are definable and connected.
If $G$ acts irreducibly on $V$, one also says that $V$ is $G$-minimal.

\renewcommand{\theoremname}{Zilber's Field Theorem}
\begin{generictheorem*}[{from \cite[Theorem 9.1]{BNGroups}}]
Let $G = A \rtimes H$ be a group of finite Morley rank where $A$ and $H$ are infinite definable abelian subgroups and $A$ is $H$-minimal. Assume $C_H(A) = 1$. Then there is a definable field structure $\K$ with $H \hookrightarrow \K^\times$ in its action on $A \simeq \K_+$ (all definably).
\end{generictheorem*}

Zilber's Field Theorem has several variants and generalisations we shall encounter in the proof of Proposition \ref{p:reductions}. But for the bulk of the argument, the original version we just gave suffices.

Here are two more results of repeated use; notice the difference of settings, since in the rank $3k$ analysis the group is supposed to be given explicitly. The present work extends the rank $2$ analysis.


\renewcommand{\theoremname}{Rank $2$ Analysis}
\begin{generictheorem*}[{\cite[Theorem A]{DActions}}]
Let $G$ be a connected, non-soluble group of finite Morley rank acting definably and faithfully on a connected abelian group $V$ of Morley rank $2$. Then there is an algebraically closed field $\K$ of Morley rank $1$ such that $V \simeq \K^2$, and $G$ is isomorphic to $\GL_2(\K)$ or $\SL_2(\K)$ in its natural action.
\end{generictheorem*}

\renewcommand{\theoremname}{Rank $3k$ Analysis}
\begin{generictheorem*}[{\cite{CDSmall}}]
In a universe of finite Morley rank, consider the following definable objects: a field $\K$, a group $G \simeq \pSL_2 (\K)$, an abelian group $V$, and a faithful action of $G$ on $V$ for which $V$ is $G$-minimal. Assume $\rk V \leq 3 \rk \K$. Then $V$ bears a structure of $\K$-vector space such that:
\begin{itemize}
\item
either $V \simeq \K^2$ is the natural module for $G \simeq \SL_2 (\K)$,
\item
or $V \simeq \K^3$ is the irreducible $3$-dimensional representation of $G \simeq \PSL_2(\K)$ with $\mathrm{char}\ \K \neq 2$.
\end{itemize}
\end{generictheorem*}

In particular, $\SL_2(\K)$ acting on an abelian group of rank $3$ must centralise a rank $1$ factor in a composition series; in characteristic not $2$, composition series then split thanks to the central involution.

\subsection{Two Trivial Generalities}

Here are two principles no one cared to write down so far.

\begin{lemma}
Let $T$ be a good torus acting definably and faithfully on a module $V$. Then $\rk T \leq \rk V$, and for any prime $q$ with $U_q(V) = 1$:
\[\rk T \leq \rk V + \Pr\nolimits_q(T) - \ell_T(V)\]
\end{lemma}
\begin{proof}
Induction on $\rk V$.
The result is obvious if $\rk V = 0$. So let $0 \leq W < V$ be such that $V/W$ is $T$-minimal, and set $\Theta = C_T(V/W)$. Notice that $\Theta\o$ is a good torus and acts faithfully on $W$; one has $\ell_{\Theta\o}(W) \geq \ell_T(W) = \ell_T(V) - 1$. So by induction,
\[\rk (\Theta\o) \leq \rk W + \Pr\nolimits_q(\Theta\o) - \ell_{\Theta\o}(W) \leq \rk W + \Pr\nolimits_q(\Theta\o) - \ell_T(V) + 1\]
and therefore:
\[\rk T \leq \rk \left(T/\Theta\right) + \rk W + \Pr\nolimits_q(\Theta\o) - \ell_T(V) + 1\]
(Also bear in mind the other estimate $\rk T \leq \rk(T/\Theta) + \rk W$.)

By Zilber's Field Theorem there is a field structure $\K$ such that $T/\Theta \hookrightarrow \K^\times$ and $V/W \simeq \K_+$ definably (this is not literally true in case $\Theta = T$ as there is no field structure around, but this is harmless). Quickly notice that $\rk (T/\Theta) \leq \rk (\K^\times) = \rk (\K_+) = \rk(V/W)$, so $\rk T \leq \rk(V/W) + \rk W = \rk V$.
If $T/\Theta$ is proper in $\K^\times$, then actually $\rk (T/\Theta) \leq \rk V - \rk W - 1$ whereas $\Pr_q(\Theta\o) \leq \Pr_q(T)$: we are done.
If on the other hand $T/\Theta \simeq \K^\times$, then $\rk (T/\Theta) = \rk V - \rk W$ and $\Pr_q(\Theta\o) = \Pr_q(T) - 1$ since $\K$ does not have characteristic $q$: we are done again.
\end{proof}


\begin{lemma}
Let $H$ be a definable, connected group acting definably and faithfully on a module $V$ of exponent $p$.
If $H$ is soluble, then $H = U \rtimes T$ with $U = U_p(H)$ and $T$ a good torus.
Moreover, $H$ centralises all quotients in an $H$-composition series of $V$ if and only if $H$ is $p$-unipotent, in which case the exponent is bounded by $q = p^k$ with $q \geq \ell_H(V)$.
\end{lemma}
\begin{proof}
First suppose $H$ to be soluble. By faithfulness and the structure theorem for locally soluble $p$-groups \cite[Corollary 6.20]{BNGroups}, $H$ contains no $p$-torus. Moreover, the only unipotence parameter \cite[\S2.3]{DJInvolutive} which can occur in $H$ is $(p, \infty)$. In particular, $H/U_p(H)$ has no unipotence at all: it is a good torus. Let $T \leq H$ be a maximal good torus of $H$. Then $T$ covers $H/U_p(H)$, and $T\cap U_p(H) = 1$ since $T$ has no element of order $p$. Therefore $H = U \rtimes T$ for $T$ a maximal good torus.

If $H$ is actually $p$-unipotent, it clearly centralises all quotients in an $H$-composition series. Conversely, if $H$ centralises all quotients in $0 = V_0 < \dots < V_\ell = H$, then $H$ is soluble of class $\leq \ell - 1$: induction on $\ell$, the claim being obvious at $\ell = 1$. 
So write $H = U \rtimes T$ as above. By assumption, $T$ centralises all quotients in the series so $T$ centralises $V$;
by faithfulness, $T = 0$ and $H = U$ is $p$-unipotent. Finally observe how for $u \in U$, $(u-1)^\ell = 0$ in $\End(V)$. So for $q = p^k \geq \ell$, one has $(u-1)^q = u^q - 1 = 0$ in $\End(V)$ and $u^q = 1$ in $H$.
\end{proof}

In particular, when acting on a module of exponent $p$, decent tori \cite{CGood} of automorphisms are good tori.

\section{The Proof}

We now start proving the theorem. After an initial section (\S\ref{s:linear}) dealing with various aspects of linearity, we shall adopt a more abstract line. The main division is along values of the Pr{\"u}fer $2$-rank, a measure of the size of the Sylow $2$-subgroup. We first handle the pathological case of an acting group with no involutions, which we prove bad; configurations are tight and we doubt that any general lesson can be learnt from \S\ref{s:Pr2=0}. Then \S\ref{s:Pr2=1} deals with the Pr{\"u}fer rank $1$ case where the adjoint action of $\PSL_2(\K)$ is retrieved; this makes use of recent results on abstract, so-called $N_\circ^\circ$-groups. \S\ref{s:Pr2=2} is essentially different: when the Pr{\"u}fer rank is $2$, we can use classical group-theoretic technology, viz. strongly embedded subgroups.

\begin{notation*}\
\begin{itemize}
\item
Let $G$ be a connected, non-soluble group of finite Morley rank acting definably and faithfully on an abelian group $V$ of rank $3$ which is $G$-minimal.
\item
Let $S \leq G$ be a Sylow $2$-subgroup of $G$; if $G$ has odd type, let $T\leq G$ be a maximal good torus containing $S\o$.
\end{itemize}
\end{notation*}

Notice that we do \emph{not} make assumptions on triviality of $C_V(G)$; of course by $G$-minimality of $V$, the former is finite. For the same reason, $V$ is either of prime exponent or torsion-free; the phrase ``the characteristic of $V$'' therefore makes sense.

\subsection{Reductions}\label{s:linear}

We first deal with a number of reductions involving a wide arsenal of methods. Model-theoretically speaking we shall use two $n$-dimensional versions of Zilber's Field Theorem: \cite[Theorem 9.5]{BNGroups} which linearises irreducible actions of non semi-simple groups, in the abstract sense of $R\o(G) \neq 1$, and \cite[Theorem 4]{LWCanada}, which linearises actions on torsion-free modules. We shall also invoke work of Poizat \cite{PQuelques} generalised by Mustafin \cite{MStructure} on the structure of definably linear groups of finite Morley rank, which in characteristic $p$ is a consequence of Wagner's Torus Theorem. In a more group-theoretic direction, we shall rely on the classification of the simple groups of finite Morley rank of even type \cite{ABCSimple}, and a theorem of Timmesfeld \cite{TIdentification} on abstract $\SL_n(\K)$-modules will play a significant part.

\begin{proposition}\label{p:reductions}
We may suppose that $C_V(G) = 0$, that $R\o(G) = 1$, and that $V$ has exponent an odd prime number $p$. In particular every definable, connected, soluble subgroup $B \leq G$ has the form $B = Y \rtimes \Theta$ where $Y$ is a $p$-unipotent subgroup and $\Theta$ is a good torus (either may be trivial).
\end{proposition}

The connected soluble radical $R\o(G)$ was first studied by Belegradek; see \cite[\S7.2]{BNGroups}.

\begin{proof}
\begin{step}
We may suppose $C_V(G) = 0$.
\end{step}
\begin{proofclaim}
Let $\overline{V} = V/C_V(G)$, which clearly satisfies the assumption. Suppose that the desired classification holds for $\overline{V}$: then $(G, \overline{V})$ is known. If $G$ is a simple bad group of rank $3$, we are done as we assert nothing on the action. If $G$ contains $\PSL_2(\K)$, we know the structure of $V$ by the rank $3k$ analysis, and $C_V(G) = 0$.
If $G$ contains $\SL_3(\K)$ acting naturally on $\overline{V}$, we show $C_V(G) = 0$ as follows.

More generally: if $\K$ is any field of finite Morley rank and $G \simeq \SL_n(\K)$ acts definably on a definable, connected module $V$ such that $C_V(G)$ is finite and $V/C_V(G)$ is the natural $G$-module, then $C_V(G) = 0$.
%

$V$ is $G$-minimal because $V/C_V(G)$ is and $C_V(G)$ is finite. In particular, if $\K$ has characteristic zero then $V$ is torsion-free and $C_V(G) = 0$. Otherwise, $V$ has prime exponent the characteristic $p$ of $\K$.
Set $W = C_V(G)$. Let $x \in V \setminus W$ and set $H = C_G(x)$. Consider the image $\overline{x}$ in $V/W$. Then by inspection, $C_G(\overline{x})$ is a semi-direct product $\K^{n-1} \rtimes \SL_{n-1}(\K)$; in particular it is connected, and has rank $(n(n-1)-1)\cdot \rk \K$. Now by Zilber's Indecomposibility Theorem, $[C_G(\overline{x}), x]$ is a connected subgroup of the finite group $W$, hence trivial: it follows that $C_G(\overline{x}) = C_G(x)$, a group we denote by $H$. Moreover, $O = x^G$ has rank $n\cdot \rk \K$ so it is generic in $V$.
By connectedness of $V$, $V \setminus O$ is not generic. Fix $w_0 \in W\setminus\{0\}$. Since $\<w_0\>$ is finite there is a translate $v + \<w_0\>$ of $\<w_0\>$ contained in $O$. Hence there are $x$ and $y$ in $V$ with $y = x + w_0$ and $y = x^g$ for some $g \in G$. Iterating, one finds $x^{g^p} = x$, so $g^p \in H \simeq \K^{n-1} \rtimes \SL_{n-1}(\K)$. But on the other hand, since $G$ centralises $w_0$, $g$ normalises $H$ (the author forgot to write down this sentence in the proof of \cite[Fact 2.7]{DActions}).
Now $g \in N_G(U_p(H))$ which is an extension of $H$ by a torus as a computation in $\SL_n(\K)$ reveals. This and $g^p \in H$ show $g \in H$, so $x = y$: a contradiction.
\end{proofclaim}

\begin{step}\label{p:reductions:st:RG}
If $G$ is definably linear (i.e. there is a field structure $\K$ such that $V \simeq \K^n$ and $G \hookrightarrow \GL(V)$, all definably), then the theorem is proved.
\end{step}
\begin{proofclaim}
Suppose that there is a definable field structure $\K$ with $V\simeq \K_+^n$ and $G \hookrightarrow \GL(V)$ definably. Then clearly $\rk \K = 1$ and $n = 3$; hence $G \leq \GL_3(\K)$ is a definable subgroup.
Be careful that a field of Morley rank $1$ need not be a pure field (see \cite{HStrongly} for the most drammatic example), so there remains something to prove.

We shall show that $G$ is a closed subgroup of $\GL_3(\K)$.
If $R\o(G) \neq 1$ then linearising again with \cite[Theorem 9.5]{BNGroups} and up to taking $\K$ to be the newly found field structure, $R\o(G) = \K^\times \Id_V$. We then go to $H = (G\cap \SL(V))\o$, which satisfies $G = H\cdot R\o(G)$, so that it suffices to show that $H$ is closed. Hence we may assume $R\o(G) = 1$.
If the characteristic is finite then by \cite[Theorem 2.6]{MStructure}, we are done. So we may assume that $V$ is torsion-free. If the definable subgroup $G \leq \GL_3(\K)$ is not closed, by \cite[Theorem 2.9]{MStructure}, we find a definable subgroup $K \leq G$ which contains only semi-simple elements, in the geometric sense of the term. We may assume that $K$ is minimal among definable, connected, non-soluble groups: it is then a bad group. But $\rk \K = 1$, so any Borel subgroup of $K$ is actually a good torus and contains involutions: this contradicts the bad group analysis. One could also argue through the unfortunately unpublished \cite{BBDefinably}.

As a consequence, $G$ is closed and therefore algebraic.
We now inspect irreducible, algebraic subgroups of $\GL_3(\K)$ to conclude.
\end{proofclaim}

\begin{step}
If $R\o(G) \neq 1$ or $V$ is torsion-free then the theorem is proved.
\end{step}
\begin{proofclaim}
If $R\o(G) \neq 1$, then we linearise the setting with \cite[Theorem 9.5]{BNGroups} and rely on Step \ref{p:reductions:st:RG}. If $V$ is torsion-free then we use \cite[Theorem 4]{LWCanada} with the same effect.
\end{proofclaim}

So we may assume that $V$ has prime exponent $p$.
As a consequence, any definable, connected, soluble subgroup $B \leq G$ has the form $B = Y \rtimes \Theta$ where $Y$ is a $p$-unipotent subgroup and $\Theta$ is a good torus.

\begin{step}
If $V$ has exponent $2$ then the theorem is proved.
\end{step}
\begin{proofclaim}
Here we draw the big guns: the even type classification \cite{ABCSimple}. Keep $R\o(G) = 1$ in mind. Let $H \leq G$ be a component, which is a quasi-simple algebraic group over a field of characteristic $2$; $H$ acts irreducibly.
Since $\SL_2(\K) \simeq \PSL_2(\K)$ has no irreducible rank $3$ module in characteristic $2$ by the rank $3k$ analysis, we know $H \not\simeq \SL_2(\K)$.
Now let $T_H \leq H$ be an algebraic torus. Then $\rk T_H \leq \rk V = 3$, so $H$ has Lie rank at most $3$.
$H$ is then a simple algebraic group of one of types $A_2$, $B_2$, $A_3$, $B_3$, $C_3$, or $G_2$. A brief look at the extended Dynkin diagrams for these groups shows that in all cases other than $A_2$, $H$ contains a subgroup of type $A_1 + A_1$, that is, a commuting product of two groups $\SL_2$, which is obviously impossible. So $H$ has type $A_2$.
But $T_H$ extends to a maximal good torus of $G$, still of rank $\leq 3$, and there are therefore no other components.
As a consequence, $G = H \simeq {\rm (P)SL}_3(\K)$.

It remains to identify the action. We rely on work by Timmesfeld \cite{TIdentification}.

Let $U_0 \leq G$ be a root subgroup, say $U_0 \leq G_0 \simeq \SL_2(\K)$. We claim that $\rk [V, U_0] = 1$. Let $T_1$ be a one-dimensional torus centralising $U_0$, say $T_1 \leq G_1 \simeq \SL_2(\K)$. Now by the rank $3k$ analysis, $G_1$ cannot act irreducibly, so there is a $G_1$-composition series for $V$ where $G_1$ centralises the rank $1$ factor. Hence $C_V(T_1) \neq 0$, and the action of $U_0$ on $V = C_V(T_1) \oplus [V, T_1]$ shows $\rk [V, U_0] = 1$.

As a consequence, if $U_1 \leq C_G(U_0)$ is another root subgroup, then $U_1$ centralises the rank $1$ subgroup $[V, U_0]$, meaning $[V, U_0, U_1] = 0$. So we are under the assumptions of \cite{TIdentification} and conclude that $G \simeq \SL_3(\K)$ acts on $V\simeq \K^3$ as on its natural module.
\end{proofclaim}

This concludes our series of reductions.
\end{proof}

We finish these preliminaries with a quick remark.

\begin{definition*}[and Observation]
A definable, connected subgroup $V_1 \leq V$ is called a \emph{TI} subgroup (for: Trivial Intersections) if $V_1\cap V_1^g = 0$ for all $g \notin N_G(V_1)$. When this holds, $\cork N_G(V_1) = \rk (G/N_G(V_1)) \leq 2$.
\end{definition*}
\begin{proof}
Consider the family $\mathcal{F} = \{V_1^g: g \in G\}$: its rank is $\cork N_G(V_1)$. The TI assumption means that elements of the family are pairwise disjoint, so $\rk \bigcup_{\mathcal{F}} = \rk \mathcal{F} + \rk V_1 \leq \rk V$ and $\cork N_G(V_1) \leq 2$.
\end{proof}

The rest of the proof is a case division along the Pr\"{u}fer $2$-rank of $G$. It is much more group-theoretic, and much less model-theoretic, in nature.

\subsection{The Pr\"{u}fer Rank 0 Analysis}\label{s:Pr2=0}

We now deal with desperate situations: if $G$ has no involutions, then it is a simple bad group of rank $3$ (Proposition \ref{p:Pr2=0}). If it has, then it has a Borel subgroup of mixed nature $\beta = Y \rtimes \Theta$ (Proposition \ref{p:Pr2>0:mixedBorel}). Although it will not be clear at that point whether $\beta$ need be non-nilpotent, it will serve as a \emph{deus ex machina} in Step \ref{p:Pr2=1:PSL2:st:T} of Proposition \ref{p:Pr2=1:PSL2}.

\subsubsection*{More Material}

The main ingredients in this section are Hrushovski's Theorem on strongly minimal actions, the analysis of bad groups, and Wiscons' analysis of groups of rank $4$. But uniqueness principles in $N_\circ^\circ$-groups also play a key role.

\renewcommand{\theoremname}{Wiscons' Analysis}
\begin{generictheorem*}[from {\cite[Corollary A]{WRank4}}]
If $G$ is a connected group of rank $4$ with involutions then $F\o(G) \neq 1$.
\end{generictheorem*}

Recall from \cite{DJInvolutive} that a group of finite Morley rank $G$ is an $N_\circ^\circ$-group if for any infinite, definable, connected, abelian subgroup $A \leq G$, the connected normaliser $N_G\o(A)$ is soluble. Our theorem is the second application of the theory of $N_\circ^\circ$-groups after \cite{DActions}. However, because of Proposition \ref{p:reductions}, only the rather straightforward, positive version of uniqueness principles will be used; Burdges' subtle unipotence theory \cite[\S2.3]{DJInvolutive} will not.

\renewcommand{\theoremname}{Uniqueness Principles in $N_\circ^\circ$-groups}
\begin{generictheorem*}[from {\cite[Fact 8]{DJInvolutive}}]
Let $G$ be an $N_\circ^\circ$-group and $B$ be a Borel subgroup of $G$.
Let $U \leq B$ be a non-trivial, $p$-unipotent subgroup of $B$. Then $B$ is the only Borel subgroup of $G$ containing $U$.
\end{generictheorem*}

It may be good to keep in mind that if $G$ is an $N_\circ^\circ$-group and $B$ is a Borel subgroup with $U_p(B) \neq 1$, then $U_p(B)$ is actually a \emph{maximal} $p$-unipotent group of $G$.

\subsubsection{Bad groups}

\begin{proposition}\label{p:Pr2=0}
If $G$ has no involutions, then $G$ is a simple bad group of rank $3$.
\end{proposition}
\begin{proof}
By the rank $2$ analysis and since there are no involutions, any definable, connected, reducible subgroup is soluble. Let $A \leq G$ be a non-trivial, definable, connected, abelian subgroup. If $N = N_G\o(A) < G$ is irreducible then by induction $N$ can only be a bad group of rank $3$, a contradiction. Hence $N$ is reducible and therefore soluble. As a consequence $G$ is an $N_\circ^\circ$-group and we shall freely use uniqueness principles in Steps \ref{p:Pr2=0:st:Bnilp} and \ref{p:Pr2=0:st:Uneq1OK}.

\begin{step}\label{p:Pr2=0:st:Cv0}
For $v_0 \in V\setminus\{0\}$, $C_G\o(v_0)$ is a soluble group of corank $2$.
\end{step}
\begin{proofclaim}
Observe how $\bigcap_{g \in G} C_G\o(v_0)^g \leq C_G(\<v_0^G\>) = C_G(V) = 0$ by faithfulness. So if the corank of $H = C_G\o(v_0)$ is $1$ we apply Hrushovski's Theorem to the action of $G$ on $G/H$ and find $G \simeq \PSL_2(\K)$, a contradiction to the absence of involutions. If on the other hand $\cork H = 3$ then $v_0^G$ is generic in $V$, and so is $-v_0^G$: lifting torsion, this creates an involution in $G$, a contradiction again.

Now suppose that $H$ is non-soluble: it is therefore irreducible, so by induction it is a bad group of rank $3$. In particular $\rk G = 5$; always by Hrushovski's Theorem, $N_G\o(H) = H$. Hence $\{H^g: g \in G\}$ has rank $2$ and degree $1$. However for $g \notin N_G(H)$, $H\cap H^g$ has rank $1$, so ${\rm Stab}_{H}(H^g)$ has rank $1$. Therefore all orbits in the action of $H$ on $\{H^g: g \notin N_G(H)\}$ are generic: the action is transitive. This shows that $G$ acts $2$-transitively on $\{H^g: g \in G\}$, and lifting torsion there is an involution in $G$: a contradiction.
\end{proofclaim}

\begin{notationinproof*}
Let $B = Y \rtimes \Theta$ be a Borel subgroup, with $Y$ a $p$-unipotent subgroup and $\Theta$ a good torus (either term or the action may be trivial).
\end{notationinproof*}

\begin{step}\label{p:Pr2=0:st:Bnilp}
$Y$ or $\Theta$ is trivial.
\end{step}
\begin{proofclaim}
Suppose $Y \neq 1$ and $\Theta\neq 1$. Then $\Theta$ acts on $C_V\o(Y) \neq 0$ so $V$ is not $\Theta$-minimal.
In a $\Theta$-composition series there is therefore a $\Theta$-invariant subquotient module of $V$ of rank $1$, say $X_1$. By Zilber's Field Theorem and since $G$ has no involutions, $\Theta$ centralises $X_1$, and this shows $C_V(\Theta) \neq 0$. Hence $V = C_V(\Theta) \oplus [V, \Theta]$ is a non-trivial decomposition. Again with Zilber's Field Theorem and the absence of involutions, $[V, \Theta]$ has rank $2$ and $\Theta$ has rank $1$. Always for the same reasons, $Y\rtimes \Theta$ is now nilpotent and $\Theta$ centralises $Y$. So $Y$ normalises both $C_V(\Theta)$ and $[V, \Theta]$, and it follows $\rk C_V\o(Y) \geq 2$. Then for $g\notin N_G(Y)$ the group $\<Y, Y^g\>$ is reducible, therefore soluble, forcing $Y = Y^g$: a contradiction.
\end{proofclaim}

\begin{step}\label{p:Pr2=0:st:Uneq1OK}
If $Y \neq 1$, then $G$ is a simple bad group of rank $3$.
\end{step}
\begin{proofclaim}
Let $V_1 = C_V\o(Y) \neq 1$.
If $\rk V_1 = 2$ then for $g \notin N_G(Y)$ the group $\<Y, Y^g\>$ is reducible, therefore soluble, which forces $Y = Y^g$, a contradiction. Hence $\rk V_1 = 1$. Suppose that $V_1$ is not TI: there are $g \notin N_G(V_1)$ and $v_1 \in V_1\cap V_1^g \setminus\{0\}$. Then $H = C_G\o(v_1) \geq \<Y, Y^g\>$ is soluble by Step \ref{p:Pr2=0:st:Cv0}, which yields the same contradiction.

We have just proved that $V_1$ is a rank $1$, TI subgroup.
But $Y \leq N_G\o(V_1)$, and equality follows since $N_G\o(V_1)$ is soluble and $Y$ is a Borel subgroup by Step \ref{p:Pr2=0:st:Bnilp}.
So $\cork Y \leq 2$; by uniqueness principles, $2 \rk Y \leq \rk G \leq \rk Y + 2$ and $\rk G \leq 4$. As a matter of fact, Wiscons' work \cite{WRank4} rules out equality; let us give a quick argument. If $\rk Y = 2$ and $\rk G = 4$, then exactly like in Step \ref{p:Pr2=0:st:Cv0}, $Y$ is transitive on $\{Y^g: g \notin N_G(Y)\}$, which creates an involution in $G$. Hence $\rk G = 3$.

It remains to prove simplicity. Observe that $G$ has no good torus since (for instance) Borel subgroups of $G/Z(G)$ are conjugate by the bad group analysis. So torsion in $G$ consists of $p$-elements. Now if $\alpha \in Z(G)$ then $C_V\o(\alpha) \neq 0$, contradicting $G$-minimality.
\end{proofclaim}

\begin{step}
If $G$ has no unipotent subgroup, then $G$ is a simple bad group of rank $3$.
\end{step}
\begin{proofclaim}
Suppose that $V$ is $\Theta$-minimal: then by Zilber's Field Theorem $\Theta$ acts freely on $V$. Let $v_0 \in V\setminus\{0\}$. By Step \ref{p:Pr2=0:st:Cv0}, $C_G\o(v_0)$ is soluble, therefore a good torus. By the conjugacy of maximal good tori we may assume $C_G\o(v_0) \leq \Theta$, against freeness of $\Theta$.

This shows that $V$ is not $\Theta$-minimal. Like in Step \ref{p:Pr2=0:st:Bnilp}, $C_V(\Theta) \neq 0$, $\rk [V, \Theta] = 2$, and $\rk \Theta = 1$. Let $v_0 \in C_V(\Theta)$; then $\Theta \leq C_G\o(v_0)$ but $\Theta$ is a Borel subgroup and $C_G\o(v_0)$ is soluble by Step \ref{p:Pr2=0:st:Cv0}: this shows $\Theta = C_G\o(v_0)$ and $\rk G = 3$.

It remains to prove simplicity. If there is $\alpha \in Z(G)$ then up to taking a power, $\alpha$ has prime order $q \neq p$. By torality principles, $\alpha \in \Theta$: hence $C_V\o(\alpha) \neq 0$, contradicting $G$-minimality.
\end{proofclaim}

This concludes the Pr\"{u}fer rank $0$ analysis.
\end{proof}

\subsubsection{Good groups}

From now on we shall suppose that $G$ has involutions. It follows easily that $G$ is not bad (this is done in the proof below); yet it is not clear at all whether $G$ has a non-nilpotent Borel subgroup. For the moment one could imagine that all proper, non-soluble subgroups of $G$ are bad of rank $3$, with $G$ simple. (Recall that a bad group is defined by the condition that \emph{all} definable, connected, proper subgroups are nilpotent: not only the soluble ones.) We nonetheless push a little further towards non-badness. Recall that we had let $S \leq G$ be a Sylow $2$-subgroup: in view of Proposition \ref{p:reductions} and the current assumption, $S\o$ is a $2$-torus; we had also let $T \leq G$ be a maximal good torus containing $S\o$.

\begin{proposition}\label{p:Pr2>0:mixedBorel}
Suppose that $G$ has an involution. Then $G$ has a Borel subgroup $\beta = Y \rtimes \Theta$ where $Y \neq 1$ is a non-trivial $p$-unipotent group and $\Theta \neq 1$ is a non-trivial good torus (but the action may be trivial). Moreover $V$ as a $T$-module has length $\ell_T (V) \geq 2$.
\end{proposition}
\begin{proof}
We address the first claim; the second one will be proved in the final Step \ref{p:Pr2>0:mixedBorel:st:length}. Suppose that $G$ has no such Borel subgroup. Then all definable, connected, soluble subgroups are nilpotent, and therefore by the rank $2$ analysis all definable, connected, non-soluble, proper subgroups are irreducible, so by induction they are simple bad groups of rank $3$. It also follows that $G$ is an $N_\circ^\circ$-group; we shall use uniqueness principles.

\begin{step}
$G$ has no unipotent subgroup.
\end{step}
\begin{proofclaim}
Suppose it does, and let $U \neq 1$ be maximal as such. By assumption, $U$ is a Borel subgroup of $G$. Now for $g\notin N_G(U)$, $U\cap U^g = 1$. Otherwise, there is $x \in U\cap U^g \setminus\{1\}$. But $x$ is a $p$-element, so $C_G\o(x)$ is reducible and therefore soluble, and it contains $Z\o(U)$ and $Z\o(U^g)$. This contradicts uniqueness principles in $N_\circ^\circ$-groups.

As a consequence, $U$ is disjoint from its distinct conjugates and of finite index in its normaliser, therefore $U^G$ is generic in $G$. By \cite[Theorem 1]{BCSemisimple}, the definable hull $d(u)$ of the generic element $u \in U$ now contains a maximal $2$-torus: a contradiction.
\end{proofclaim}

It follows that $T$ is a Borel subgroup.

\begin{step}\label{p:Pr2>0:mixedBorel:st:Theta}
$T$ contains a good torus $\Theta$ of rank $1$ with no involutions.
\end{step}
\begin{proofclaim}
Quickly notice that $G$ itself is not bad. If it is, then by the bad group analysis and since there are involutions, $G$ is not simple: there is an infinite, proper, normal subgroup $N \triangleleft G$; since $G$ is bad, $N$ is nilpotent, against Proposition \ref{p:reductions}.

Hence $G$ is not bad.
By definition there is a definable, connected, non-nilpotent, proper subgroup $H < G$: $H$ is non-soluble, hence a bad group of rank $3$. Let $\Theta < H$ be a Borel subgroup of $H$: since $G$ has no unipotent elements, $\Theta$ is a good torus of rank $1$, and has no involutions.

By the conjugacy of maximal good tori in $G$ we may assume $\Theta \leq T$; inclusion is proper since $T$ does have involutions.
\end{proofclaim}

\begin{step}
$V$ is not $T$-minimal and $\rk T = 2$.
\end{step}
\begin{proofclaim}
If $V$ is $T$-minimal, then by Zilber's Field Theorem $T$ acts freely. Now for $v_0 \in V\setminus\{0\}$, $C_G\o(v_0)$ contains neither unipotent, nor toral subgroups: by Reineke's Theorem it is trivial and $\rk G = 3$. Now $G$ is a quasi-simple bad group of rank $3$ but it contains an involution: against the bad group analysis. So $V$ is not $T$-minimal, $\ell_T(V) \geq 2$; since $\rk (V) = 3$ and $\Pr_2(T) = 1$, we deduce $\rk T \leq 2$.
\end{proofclaim}

\begin{step}
$V_1 = C_V(\Theta)$ has rank $1$ and $V_2 = [V, \Theta]$ has rank $2$. There is a field structure $\L$ with $V_2 \simeq \L_+$ and $\Theta < \L^\times$.
\end{step}
\begin{proofclaim}
$V$ is not $T$-minimal, so it is not $\Theta$-minimal either.
Notice that $\Theta$ having no involutions, must centralise rank $1$ subquotient modules by Zilber's Field Theorem. It follows $V = C_V(\Theta) \oplus [V, \Theta]$ where $V_1 = C_V(\Theta)$ has rank $1$ and $V_2 = [V, \Theta]$ has rank $2$ and is $\Theta$-minimal. Apply Zilber's Field Theorem again to get the desired structure.
\end{proofclaim}

\begin{step}\label{p:Pr2>0:mixedBorel:st:V1Ttrivial}
If $T$ does not centralise $V_1$, then we are done.
\end{step}
\begin{proofclaim}
Suppose that $T$ does not centralise $V_1$, meaning $C_T\o(V_1) = \Theta$. By Zilber's Field Theorem there is a field structure $\K$ with $V_1 \simeq \K_+$ and $T/\Theta \simeq T/C_T(V_1) \simeq \K^\times$. But $\Theta$ is a non-trivial good torus, so there is a prime number $q \neq 2$ with $\Pr_q(\Theta) = 1$, showing $\Pr_q(T) \geq 2$. In particular, $T$ does not embed into $\L^\times$, so $\tau = C_T\o(V_2)$ is infinite. Since $\tau \cap \Theta = 1$, one finds $T = \Theta \times \tau$. Finally let $v_2 \in V_2 \setminus\{0\}$ and $K = C_G\o(v_2) \geq \tau$. If $K$ is non-soluble, then it is a bad group of rank $3$, a contradiction since $\tau$ has involutions. So $K$ is soluble and by the structure of Borel subgoups, $K \leq T$. Since $\Theta$ acts freely on $V_2$, $K = \tau$ has corank $\leq 3$, and $G$ has rank $\leq 4$. By Wiscons' analysis, $F\o(G) \neq 1$: against Proposition \ref{p:reductions}.
\end{proofclaim}

\begin{step}
If $T$ centralises $V_1$, then we are done.
\end{step}
\begin{proofclaim}
Now suppose instead that $T$ centralises $V_1$. Observe how $C_V(T) = V_1$ and $N_G\o(V_1) = T$ by solubility of the former and maximality of the latter as a definable, connected, soluble group; in particular $N_G(T) = N_G(V_1)$. If $V_1$ is not TI, then there are $g \notin N_G(V_1)$ and $v_1 \in V_1\cap V_1^g \setminus\{0\}$; now $K = C_G\o(v_1) \geq \<T, T^g\>$ is non-soluble and therefore a bad group of rank $3$, a contradiction to $\rk T = 2$. Hence $V_1$ is TI, proving $\cork T \leq 2$ and $\rk G \leq 4$. Finish like in Step \ref{p:Pr2>0:mixedBorel:st:V1Ttrivial}.
\end{proofclaim}

We have proved the main statement; it remains to study the length of $V$ as a $T$-module.

\begin{step}\label{p:Pr2>0:mixedBorel:st:length}
Consequence: $V$ is not $T$-minimal.
\end{step}
\begin{proofclaim}
Suppose it is. Then by Zilber's Field Theorem there is a field structure $\L$ with $V \simeq \L_+$ and $T\leq \L^\times$. Let $\beta = Y \rtimes \Theta$ be a Borel subgroup of mixed structure, and consider $W = C_V\o(Y) \neq 0$. Then $\Theta$ normalises $W$ and $V/W$, and one of them, say $X_1$, has rank $1$. By freeness of toral elements and Zilber's Field Theorem, there is a definable field structure $\K$ with $X_1 \simeq \K_+$ and $\Theta \simeq \K^\times$. Hence $\K^\times \simeq \Theta \leq T \leq \L^\times$, and $T/\Theta$ is torsion-free. Now Wagner's Torus Theorem forces $T = \Theta$: so $V$ is not $T$-minimal.
\end{proofclaim}

The proposition is proved.
\end{proof}

Our Borel subgroup $\beta$ will play a key role in Step \ref{p:Pr2=1:PSL2:st:T} of Proposition \ref{p:Pr2=1:PSL2} below.
On the other hand, we still do not know whether there is a non-nilpotent Borel subgroup. The obstacle lies in the possibility for $G$ to contain a ``bad unipotent centraliser'', we mean a bad group $K = C_G\o(v_0)$ of rank $3$ with unipotent type, in Step \ref{p:Pr2>0:mixedBorel:st:Theta} of Proposition \ref{p:Pr2>0:mixedBorel} above.
The spectre of bad groups will be haunting the Pr\"{u}fer rank $1$ analysis hereafter (and notably Proposition \ref{p:Pr2=1:PSL2}), but we are done with pathologically tight configurations.

\subsection{The Pr\"{u}fer Rank 1 Analysis}\label{s:Pr2=1}

This section is devoted to the adjoint representation of $\PSL_2(\K)$ (Proposition \ref{p:Pr2=1:PSL2}); with an early interest in \S\ref{s:Pr2=2} we shall do slightly more (Proposition \ref{p:LORO:Pr2=1}).

\subsubsection*{More Material}

The classification of $N_\circ^\circ$-groups will be heavily used throughout this section, except in Proposition \ref{p:Pr2=1:LORO} where uniqueness principles will nonetheless give the \emph{coup de gr\^{a}ce}.

\renewcommand{\theoremname}{$N_\circ^\circ$ Analysis}
\begin{generictheorem*}[{from \cite{DJInvolutive}}]
Let $G$ be a connected, non-soluble, $N_\circ^\circ$-group of finite Morley rank of odd type and suppose $G \not \simeq \PSL_2(\K)$.
Then the Sylow $2$-subgroup of $G$ is isomorphic to that of $\PSL_2(\C)$, is isomorphic to that of $\SL_2(\C)$, or is a $2$-torus of Pr\"{u}fer $2$-rank at most $2$.

Suppose in addition that for all $i \in G$, $C_G\o(i)$ is soluble.
Then involutions are conjugate and for $i \in I(G)$, $C_G\o(i)$ is a Borel subgroup of $G$. If $i \neq j$ in $I(G)$, then $C_G\o(i) \neq C_G\o(j)$.
\end{generictheorem*}

Of course one could imagine a more direct proof, reproving the necessary chunks of \cite{DJInvolutive} in the current, particularly nice context where the structure of soluble groups is very well understood.

\subsubsection{$N_\circ^\circ$-ness and Bounds}

We start with a proposition that will be used only in higher Pr\"{u}fer rank (\S\ref{s:Pr2=2}).

\begin{proposition}\label{p:LORO:Pr2=1}
If $G$ is an $N_\circ^\circ$-group then $\Pr_2(G) = 1$.
\end{proposition}
\begin{proof}
Suppose the Pr\"{u}fer rank is $\geq 2$.
By the $N_\circ^\circ$ analysis, the Sylow $2$-subgroup of $G$ is isomorphic to $\Z_{2^\infty}^2$.
In particular, since the Sylow $2$-subgroup of $G$ is connected, $G$ has no subquotient isomorphic to $\SL_2(\K)$ (see \cite[Lemma L]{DJInvolutive} if necessary): by the rank $2$ analysis, every definable, connected, reducible subgroup is soluble.

\begin{notationinproof*}
Let $\{i, j, k\}$ be the involutions in $S = S\o$.
\end{notationinproof*}


\begin{step}\label{p:LORO:Pr2=1:st:allmeek}
For $\ell = i, j, k$, $C_G\o(\ell)$ is soluble.
\end{step}
\begin{proofclaim}
Call a $2$-element $\zeta \in G$ \emph{meek} if $C_G\o(\zeta)$ is soluble; a systematic study of meek elements will be carried in the Pr\"{u}fer rank $2$ analysis (Proposition \ref{p:Pr2=2:classical}). Suppose there is a non-meek involution $i$.

Bear in mind that for any $2$-element $\zeta \in G$, $C_G\o(\zeta)$ has Pr\"{u}fer rank $2$ by torality principles. So restricting ourselves to connected centralisers of $2$-elements whenever they are non-soluble and proper, we descend to a definable, connected, non-soluble subgroup $H \leq G$ with $\Pr_2(H) = 2$ and such that every $2$-element in $H$ is either meek or central in $H$.
Since we are after a contradiction and $H$ remains irreducible on $V$ by non-solubility, we may suppose $G = H$. Since $i$ is now central in $G$, it inverts $V$; so $j$ and $k$ are therefore meek.


Since $G$ is a non-soluble $N_\circ^\circ$-group, one has $Z\o(G) = 1$: there are finitely many non-meek elements in $S$. Take one of maximal order and $\alpha \in S$ be a square root. Notice that $\alpha^2 \neq 1$ since $i$ is not meek. By construction $\alpha^2$ is central in $G$, and any element of the same order as $\alpha$ is meek: this applies to $j\alpha$ since $\alpha^2 \neq 1$.
Let us factor out $\<\alpha^2\>$ (possibly losing the action on $V$); let $\overline{G} = G/\<\alpha^2\>$ and denote the projection map by $\pi$. Observe that by Zilber's Indecomposibility Theorem and finiteness of $\<\alpha^2\>$, one has for any $g \in G$:
\[\pi^{-1}\left(C_{\overline{G}}\o(\overline{g})\right)\o = C_G\o(g)\]
In particular $\overline{j} = \overline{k}$ remains a meek involution of $\overline{G}$, and $\overline{\alpha} \neq \overline{j\alpha}$ have become meek involutions as well. So in $\overline{G}$, all involutions are meek.

By the $N_\circ^\circ$ analysis, $\overline{\alpha}$ is then $\overline{G}$-conjugate to $\overline{j}$ by some $\overline{w}$. Lifting to an element $w \in G$, one sees that $j^w = \alpha z$ for some $z\in \<\alpha^2\>$. Now $\alpha^2 = z^{-2}$ and this proves that $\alpha$ actually has order $2$: a contradiction.
\end{proofclaim}

\begin{step}
Contradiction.
\end{step}
\begin{proofclaim}
By the $N_\circ^\circ$ analysis, $B_i = C_G\o(i)$ is a non-nilpotent Borel subgroup. So $B_i$ contains some non-trivial $p$-unipotent subgroup $U_i = U_p(B_i)$. The involution $i \notin Z(G)$ centralises $U_i$, so $U_i$ normalises $V^{+_i}$ and $[V, i]$ where both are non-trivial, and this forces $\rk C_V\o(U_i) = 2$. Now for arbitrary $g$, $\<U_i, U_i^g\>$ is reducible, hence soluble: by uniqueness principles $U_i = U_i^g$, a contradiction.
\end{proofclaim}

This completes the proof.
\end{proof}

Let us repeat that Proposition \ref{p:LORO:Pr2=1} will be used only in the Pr\"{u}fer rank $2$ analysis, \S\ref{s:Pr2=2}.

\subsubsection{Bounds and $N_\circ^\circ$-ness}

The next proposition is a converse to Proposition \ref{p:LORO:Pr2=1} and the real starting point of the Pr\"{u}fer rank $1$ analysis. Notice that it does \emph{not} use the $N_\circ^\circ$ analysis, though uniqueness principles add the final touch.

\begin{proposition}\label{p:Pr2=1:LORO}
Suppose that $\Pr_2(G) = 1$. Then any definable, connected, reducible subgroup is soluble; in particular $G$ is an $N_\circ^\circ$-group.
\end{proposition}
\begin{proof}

\begin{step}\label{p:Pr2=1:LORO:st:H}
Any definable, connected, reducible, non-soluble subgroup $H \leq G$ has the form $H = U \rtimes C$, where $C \simeq \SL_2(\K)$ and the central involution $i \in C$ inverts the $p$-unipotent group $U$; $\rk H \neq 4, 6$.
Moreover if $H$ has a rank $1$ submodule $V_1$ then $V_1 = V^{+_i} = C_V\o(H)$; if $H$ has a rank $2$ submodule $V_2$ then $U$ centralises $V_2 = V^{-_i}$.
\end{step}
\begin{proofclaim}
By non-solubility of $H$ the length of $V$ as an $H$-module is $\ell_H(V) = 2$; the argument, if necessary, is as follows. Suppose $\ell_H(V) = 3$. Then all factors in a composition series are minimal, so by \cite[Proposition 3.12]{PStable} for instance, $H'$ centralises them all. Then $H$ is clearly soluble: a contradiction.

So there is an $H$-composition series $0 < W < V$; let $X_1$ be the rank $1$ factor and $X_2$ likewise; set $U = C_H(X_2)$. Then by the rank $2$ analysis, $H/U \simeq \SL_2(\K)$. Before proceeding we need to handle connectedness of $U$: it follows from the non-existence of perfect central extensions of $\SL_2(\K)$ \cite[Theorem 1]{ACCentral} by considering the isomorphisms $(H/U\o)/(U/U\o) \simeq H/U \simeq \SL_2(\K)$. Hence $U = U\o$.

As a consequence of Zilber's Field Theorem, $U$ which has no involutions must centralise $X_1$: $[V, U, U] = 1$ so $U$ is abelian. Moreover, for $u \in U$ and $v \in V$ there is $w \in W$ with $v^u = v + w$. It follows $v^{u^p} = v + pw = v$ and $U$ has exponent $p$.

Now let $i \in H$ be a $2$-element lifting the central involution in $\SL_2(\K)$: since $U$ has no involutions, $i$ is a genuine involution in $H$. Since both $U$ and $(i \mod U) \in H/U \simeq \SL_2(\K)$ centralise $X_1$, $i$ centralises $X_1$; whereas since $U$ centralises $X_2$ and $(i \mod U)$ inverts it, $i$ inverts $X_2$. We then find a decomposition $V = V^+ \oplus V^-$ under the action of $i$, with $\rk V^+ = 1$ and $\rk V^- = 2$.
\begin{itemize}
\item
If $W = X_1 \leq V$ then $U$, $H/U$, and therefore $H$ as well centralise $W$ so $V^+ = W$. For $u \in U$ and $v_- \in V^-$, there is $v_+ \in V^+$ with $v_-^u = v_- + v_+$, so $v_-^{ui} = -v_- + v_+$ and $v_-^{uiui} = v_- - v_+ + v_+ = v_-$: $uiui$ centralises $V_+ + V_- = V$. (Incidently, in this case, $H$ centralises $W$; by non-solubility of $H$, $\rk C_V\o(H) = 1$.)
\item
If $W = X_2 \leq V$ then $U$ centralises $W$ and $i$ inverts $W = V^-$. For $u \in U$ and $v_+ \in V^+$, there is $v_- \in V^-$ with $v_+^u = v_+ + v_-$, so $v_+^{ui} = v_+ - v_-$ and $v_+^{uiui} = v_+$, so $uiui$ centralises $V_+ + V_- = V$. (Incidently, in this case, $U$ centralises $W$.)
\end{itemize}
In either case, $i$ inverts $U$. All involutions of $H$ are equal modulo $U$ and $i$ inverts the $2$-divisible group $U$, so for $h \in H$, $i^h \in i U = i^U$ and $H = U \cdot C_H(i) = U \rtimes C_H(i)$. Clearly $C_H(i) = C_H\o(i) \simeq \SL_2(\K)$. Of course $\rk \K = 1$; if $\rk H \leq 6$ then by the rank $3k$ analysis and since $i$ inverts $U$, $\rk U$ must be $0$ or $2$, proving $\rk H \neq 4, 6$.
\end{proofclaim}

We start a contradiction proof. Suppose that $G$ contains a definable, connected, reducible, non-soluble group: by Step \ref{p:Pr2=1:LORO:st:H}, $G$ contains a subgroup $C \simeq \SL_2(\K)$.

\begin{notationinproof*}
Let $C \leq G$ be isomorphic to $\SL_2(\K)$ and $i \in C$ be the central involution.
\end{notationinproof*}

Before we start more serious arguments, notice that a Sylow $2$-subgroup of $C$ is one of $G$; notice further that $\ell_T(V) = 3$, so that $\rk T = 1$. Finally, by the rank $3k$ analysis, $C$ centralises $V^{+_i}$ which has rank $1$.

\begin{notationinproof*}
Let $v_+ \in V^{+_i}\setminus\{0\}$ and $v_- \in V^{-_i}\setminus\{0\}$. Set $H_+ = C_G\o(v_+)$ and $H_- = C_G\o(v_-)$.
\end{notationinproof*}

\begin{step}\label{p:Pr2=1:LORO:st:H+H-}
Both $H_+$ and $H_-$ have corank $2$ or $3$ but not both have corank $3$. Moreover, $H_+ \simeq U_+ \rtimes C$ where $U_+$ is a $p$-unipotent group inverted by $i$ and $\rk H_+ \neq 4, 6$; whereas $H_-$ is a $p$-unipotent group.
\end{step}
\begin{proofclaim}
Remember that for any $v \in V\setminus\{0\}$ one has $\bigcap_{g \in G} C_G\o(v)^g \leq C_G(\<v^G\>) = C_G(V) = 0$. So if $H_+$ or $H_-$ has corank $1$, then by Hrushovski's Theorem $G\simeq \PSL_2(\L)$, a contradiction to $G$ containing $\SL_2(\K)$. Therefore the coranks are $2$ or $3$.

Since $C$ centralises $V^{+_i}$, one has $H_+ \geq C$; by induction $H_+$ may not be irreducible, and Step \ref{p:Pr2=1:LORO:st:H} yields the desired form.

On the other hand, we claim that $H_-$ has no involutions. For if it does, say $j \in H_-$, since $i$ normalises $H_-$ and by a Frattini argument (see \cite[Lemma B]{DJInvolutive} if necessary) we may assume $[i, j] = 1$; then by the structure of the Sylow $2$-subgroup of $G$, $i = j \in H_-$ and $i$ centralises $v_-$: a contradiction.
At this point it is already clear that $v_+$ and $v_-$ may not be conjugate under $G$, and in particular that $H_+$ and $H_-$ cannot simultaneously have corank $3$ in $G$.

We push the analysis further. Suppose that $H_-$ is non-soluble. As it has no involutions, it must be irreducible by the rank $2$ analysis; by induction $H_-$ is a bad group of rank $3$, and $\rk G \leq 6$. If $\rk G = 6$ then $\rk H_+ = 4$, a contradiction. Hence $\rk G \leq 5$; on the other hand $i$ centralises $H_-$ by the bad group analysis, but $i$ is not central in $G$ since it does not invert $V$: therefore $G > C_G\o(i) \geq \<H_-, C\>$ and $\rk(H_-\cap C)\o \geq 2$, a contradiction to the structure of $H_-$. So $H_-$ is soluble. Since it has no involutions and $\rk(T) = 1$, $H_-$ is a $p$-unipotent group.
\end{proofclaim}

\begin{step}\label{p:Pr2=1:LORO:st:cork}
$\rk G \leq 6$.
\end{step}
\begin{proofclaim}
Let $x, y \in G$ be independent generic elements.

If $H_-$ centralises a rank $2$ module $V_2 \leq V$ then $(H_-\cap H_-^x)\o$ centralises $V_2 + V_2^x = V$, so $(H_-\cap H_-^x)\o = 1$ and $\rk H_- \leq \cork H_-$, proving $\rk G \leq 2\cork H_- \leq 6$.

If $H_-$ centralises a rank $1$ module $V_1 \leq V$ then $(H_-\cap H_-^x\cap H_-^y)\o = 1$ and $\rk G \leq 3 \cork H_-$; if we are not done then we may suppose $\cork H_- = 3$, and in particular $\cork H_+ = 2$.

If $H_+$ normalises a rank $2$ module $V_2 \leq V$ then we know from Step \ref{p:Pr2=1:LORO:st:H} that $V_2 = V^{-_i}$ is centralised by $U_+$. Incidently, $U_+ \neq 1$ since otherwise $\rk G \leq 6$ and we are done. But $I = (H_+\cap H_+^x)\o$ has no involutions because one such would invert $V_2 + V_2^x = V$, against the involutions in $H_+$ not being central in $G$. In particular $I$ is a unipotent subgroup of $H_+$; observe how $\rk I \geq 2 \rk H_+ - \rk G = \rk H_+ - 2 = \rk U_+ + 1$. Hence $I$ is a maximal unipotent subgroup of $H_+$, and $U_+ \leq I$. The same applies in $H_+^x$: therefore $\<U_+, U_+^x\> \leq I$, showing $(U_+ \cap U_+^x)\o \neq 1$. As the latter centralises $V_2 + V_2^x = V$, this is a contradiction.

So $H_+$ normalises a rank $1$ module $V_1 \leq V$ and by Step \ref{p:Pr2=1:LORO:st:H} again, $V_1 = V^{+_i}$ is centralised by $H_+$. In particular $(H_+\cap H_+^x\cap H_+^y) \o = 1$ and $\rk G \leq 3 \cork H_+ = 6$: we are done again.
\end{proofclaim}

\begin{step}\label{p:Pr2=1:LORO:st:contradiction}
Contradiction.
\end{step}
\begin{proofclaim}
Since $\rk G \leq 6$, one has $\rk H_+ \leq 4$; by Step \ref{p:Pr2=1:LORO:st:H}, $\rk H_+ = 3$. On the other hand $H_+ \geq C_G\o(V^{+_i}) \geq C$ shows that $H_+ = C$ does not depend on $v_+ \in V^{+_i}\setminus\{0\}$. In particular $V^{+_i}$ is TI, implying that $N = N_G\o(V^{+_i})$ has corank $\leq 2$. By Hrushovski's theorem and Proposition \ref{p:reductions}, equality holds. But $N$ is reducible and non-soluble, so by Step \ref{p:Pr2=1:LORO:st:H}, $\rk N = 3$ and $\rk G = 5$.

Now $(V^{+_i})^G$ is generic in $V$, so $v_-^G$ is not. This proves that $\rk H_- \geq 3$. But on the other hand $G$ is an $N_\circ^\circ$-group as easily seen in the current setting, so $\rk H_- \leq 2$ by uniqueness principles.
\end{proofclaim}

As a consequence, if $N = N_G\o(A)$ is non-soluble where $A \leq G$ is an infinite abelian subgroup, then $N$ is irreducible: induction and $\Pr_2(G) = 1$ yield a contradiction. This proves that $G$ is an $N_\circ^\circ$-group.
\end{proof}

\subsubsection{Identification in Pr\"{u}fer rank $1$}

We now identify the $N_\circ^\circ$ case.

\begin{proposition}\label{p:Pr2=1:PSL2}
If $\Pr_2(G) = 1$ then $G \simeq \PSL_2(\K)$ in its adjoint action on $V \simeq \K^3$.
\end{proposition}
\begin{proof}

By the rank $3k$ analysis it suffices to recognize $\PSL_2(\K)$. We wish to apply the $N_\circ^\circ$ analysis \cite{DJInvolutive}. Remember that $S$ stands for a Sylow $2$-subgroup of $G$.

\begin{notationinproof*}
Let $\alpha \in S\o$ be such that $C_G\o(\alpha)$ is soluble with $\alpha$ of minimal order. (Such an element certainly exists as $G$ is an $N_\circ^\circ$-group by Proposition \ref{p:Pr2=1:LORO}.)
\end{notationinproof*}

\begin{step}\label{p:Pr2=1:PSL2:st:alpha}
We may suppose that $C_G\o(\alpha)$ is a Borel subgroup of $G$ and $\alpha^2 \in Z(G)$.
\end{step}
\begin{proofclaim}
Let $H = C_G\o(\alpha^2)$, a non-soluble group. By Proposition \ref{p:Pr2=1:LORO}, $H$ is irreducible. If $H < G$ then by induction $H \simeq \PSL_2$; one has $\alpha^2 = 1$ and $H = G$, a contradiction.
So $G = H$ and $\alpha^2 \in Z(G)$.
We go to the quotient $\overline{G} = G/\<\alpha^2\>$, where the involution $\overline{\alpha}$ satisfies:
\[\pi^{-1}\left(C_{\overline{G}}\o(\overline{\alpha})\right)\o = C_G\o(\alpha)\]
Therefore any involution in $\overline{G}$ has a soluble-by-finite centraliser, and we apply the $N_\circ^\circ$ analysis.
If $\overline{G} \simeq \PSL_2(\K)$ then using \cite[Theorem 1]{ACCentral}, $G \simeq \PSL_2(\K)$ or $G \simeq \SL_2(\K)$; the rank $3k$ analysis brings the desired conclusion.
Therefore we may suppose that $C_{\overline{G}}\o(\overline{\alpha})$ is a Borel subgroup of $\overline{G}$, so that $C_G\o(\alpha)$ is one of $G$.
\end{proofclaim}

\begin{notationinproof*}
Let $B_\alpha = C_G\o(\alpha)$, a Borel subgroup of $G$; write $B_\alpha = U_\alpha \rtimes T$.
\end{notationinproof*}

Notice that $\alpha \in S\o \leq T$ where $T$ is the maximal good torus we fixed earlier; hence $B_\alpha$ contains $T$ all right. On the other hand it is not clear whether $B_\alpha$ is non-nilpotent, nor even whether $U_\alpha$ is non-trivial. By Proposition \ref{p:Pr2>0:mixedBorel}, non-trivial unipotent subgroups however exist.

\begin{step}\label{p:Pr2=1:PSL2:st:unipotent}
If $U \leq G$ is a maximal unipotent subgroup, then $\rk U \leq 2$ and $\rk C_V\o(U) = 1$. Moreover $\rk U_\alpha \leq 1$; if $C_V(\alpha) \neq 0$ then $U_\alpha = 1$.
\end{step}
\begin{proofclaim}
If $\rk C_V\o(U) = 2$ then for generic $g \in G$, $\<U, U^g\>$ is reducible, hence soluble by Proposition \ref{p:Pr2=1:LORO}: against maximality of $U$. Therefore $\rk C_V\o(U) = 1$, and let $V_1 = C_V\o(U)$; again, $N_G\o(V_1)$ is soluble, so we fix a Borel subgroup $B \geq N_G\o(V_1) \geq U$. Write $B = U \rtimes \Theta$ for some (possibly non-maximal) good torus $\Theta$ of $G$.

If $V_1$ is TI then $\cork B \leq 2$, but conjugates of $B$ can meet only in toral subgroups by uniqueness principles:
\[2 \rk B - \rk \Theta \leq \rk G \leq \rk B + 2\]
so $\rk U \leq 2$ and we are done.

If $V_1$ is not, then there are $g \notin N_G(V_1) \geq N_G(U)$ and $v_1 \in V_1\cap V_1^g \setminus\{0\}$; then $G > C_G\o(v_1) \geq \<U, U^g\> = K$ is non-soluble, so by induction $\rk K = 3$ and $\rk U = 1$. We are done again.

Let us review the argument in the case of $U_\alpha = U_p(C_G\o(\alpha))$, supposing $\rk U_\alpha = 2$. Then $V_\alpha = C_V\o(U_\alpha)$ is a rank $1$, TI subgroup of $V$, and $B_\alpha = N_G\o(V_\alpha)$ has corank at most $2$.

Now notice that distinct conjugates of $B_\alpha$, which may not intersect over unipotent elements by uniqueness principles, may not intersect in a maximal good torus either as otherwise $\alpha \in B_\alpha \cap B_\alpha^g$ and $B_\alpha = C_G\o(\alpha) = B_\alpha^g$. Hence $\rk (B_\alpha\cap B_\alpha^g)\o < \rk T$ and we refine our estimate into:
\[2 \rk B_\alpha - (\rk T - 1) \leq \rk G \leq \rk B_\alpha + 2\]
showing $\rk U_\alpha \leq 1$. Finally if $C_V(\alpha) \neq 0$, then $U_\alpha$ normalises $C_V(\alpha)$ and $[V, \alpha]$; this shows $\rk C_V\o(U_\alpha) \geq 2$ and forces $U_\alpha = 1$.
\end{proofclaim}

\begin{step}\label{p:Pr2=1:PSL2:st:T}
$\rk T = 1$.
\end{step}
\begin{proofclaim}
Suppose $\rk T > 1$. Then since $\Pr_2(T) =1$ and $\ell_T(V) > 1$ by Proposition \ref{p:Pr2>0:mixedBorel}, the estimate $\rk T \leq \rk V + \Pr_2(T) - \ell_T(V)$ yields $\rk T = \ell_T(V) = 2$.

We shall construct a bad subgroup of toral type; this will keep us busy for a couple of paragraphs.
In a $T$-composition series for $V$, let $X_i$ be the rank $i$ factor. Then $T \hookrightarrow T/C_T(X_1) \times T/C_T(X_2)$.

We first claim that $T$ does not centralise $X_1$. For if it does, then $V_1 = C_V(T)$ clearly has rank $1$. Now $C_V(\alpha) \neq 0$ so by Step \ref{p:Pr2=1:PSL2:st:unipotent}, $U_\alpha = 1$ and $T$ is a Borel subgroup; in view of Proposition \ref{p:Pr2=1:LORO} one has $T = N_G\o(V_1)$. If $V_1$ is TI, then $\cork T \leq 2$ and $\rk G \leq 4$; by Wiscons' analysis, the presence of involutions, and Proposition \ref{p:reductions}, this is a contradiction. Hence $V_1$ is not TI: there are $g \notin N_G(V_1) = N_G(T)$ and $v_1 \in V_1\cap V_1^g\setminus\{0\}$. Let $H = C_G\o(v_1) \geq \<T, T^g\>$; $H$ is not soluble so by Proposition \ref{p:Pr2=1:LORO} again, it is irreducible; induction yields a contradiction.
Hence $T$ does not centralise $X_1$.

We now construct a rank $1$ torus with no involutions, and prove that $T$ is a Borel subgroup.
Let $\tau = C_T\o(X_1) < T$; by Zilber's Field Theorem, there is a field structure $\K$ with $T/\tau \simeq T/C_T(X_1) \simeq \K^\times$ in its action on $X_1$. Clearly $\tau$ is a good torus of rank $1$. Since $\Pr_2(G) = 1$, $\tau$ has no involutions; since $T$ does, $\tau$ is characteristic in $T$. Now let $\tau' = C_T\o(X_2)$. If $\tau' = 1$ then by Zilber's Field Theorem again, there is a field structure $\L$ with $T \simeq T/C_T(X_2) \simeq \L^\times$ in its action on $X_2$. Then the good torus $\tau \neq 1$ has no torsion, a contradiction. Hence $\tau'$ is infinite; $T = \tau \times \tau'$ and $\tau'$ does have involutions. In particular $C_V(\alpha) \neq 0$ so by Step \ref{p:Pr2=1:PSL2:st:unipotent}, $U_\alpha = 1$ and $T$ is a Borel subgroup of $G$.

We can finally construct a bad subgroup of toral type.
Let $V_1 = C_V(\tau)$; clearly $V_1$ has rank $1$ and $N_G\o(V_1) = T$. Here again, if $V_1$ is TI then $\rk G \leq 4$, a contradiction as above. So $V_1$ is not: there are $g \notin N_G(V_1) = N_G(\tau) = N_G(T)$ and $v_1 \in V_1\cap V_1^g \setminus\{0\}$.
Let $H = C_G\o(v_1) \geq \<\tau, \tau^g\>$. If $H$ is soluble and contains no unipotence, then $H \leq C_G(\tau) = T$ and $T = T^g$, forcing $\tau = \tau^g$ and $V_1 = V_1^g$: a contradiction. If $H$ is soluble it then extends to a Borel subgroup $U \rtimes \tau$ for some non-trivial $p$-unipotent subgroup $U$. By Step \ref{p:Pr2=1:PSL2:st:unipotent}, $\rk C_V\o(U) = 1$; so $\tau$ centralises $C_V\o(U) = C_V(\tau) = V_1 = V_1^g$: a contradiction again. Hence $H$ is not soluble. By Proposition \ref{p:Pr2=1:LORO}, induction, and since $\tau$ has no involutions, $H$ is a simple bad group of rank $3$ containing toral elements.

But by Proposition \ref{p:Pr2>0:mixedBorel} there is a Borel subgroup $\beta = Y \rtimes \Theta$ where neither is trivial. Then certainly $\rk \Theta = 1$; moreover, by Step \ref{p:Pr2=1:PSL2:st:unipotent}, $W_1 = C_V\o(Y)$ has rank $1$. If $W_1$ is TI then $\cork \beta \leq 2$, so $\rk G \leq \rk Y + \rk \Theta + 2 \leq 5$. By Wiscons' analysis, $\rk G = 5$ and $\rk Y = 2$, so $\beta$ intersects $H$, necessarily in a conjugate of $\Theta$. Hence $\Theta$ has no involutions, and therefore centralises $W_1$; one sees $V = W_1 \oplus [V, \Theta]$ with $W_1 = C_V(\Theta)$. Therefore $T$ normalises $W_1$, so $N_G\o(W_1) = \beta \geq T$, a contradiction.

As a conclusion $W_1$ is not TI: there are $\gamma \notin N_G(W_1) = N_G(Y)$ and $w_1 \in W_1\cap W_1^\gamma\setminus\{0\}$. Now $K = C_G\o(w_1) \geq \<Y, Y^g\>$ is a simple bad group of rank $3$ containing unipotent elements.
Since $H \cap K = 1$, $\cork H \geq \rk K = 3$ and vice-versa. So both $v_1^G$ and $w_1^G$ are generic in $V$: they intersect, which conjugates $H$ to $K$, a contradiction.
\end{proofclaim}

Always by Proposition \ref{p:Pr2>0:mixedBorel}, there is a Borel subgroup of mixed structure $\beta = Y \rtimes \Theta$. So $T = \Theta$ itself is no Borel subgroup; in particular $U_\alpha \neq 1$ and $T = d(S\o)$.

\begin{step}\label{p:Pr2=1:PSL2:st:S=So}
$S = S\o$.
\end{step}
\begin{proofclaim}
If $S\o < S$ then there is an element $w$ inverting $S\o$; $w$ inverts $T$ as well. Let $V_1$ be a $T$-minimal subgroup of $V$. If $V_1 = V$ then $w$ gives rise to a finite-order field automorphism on $V_1 \rtimes T$: a contradiction. If $\rk V_1 = 2$ then $V_1 \cap V_1^w$ is infinite, so by $T$-minimality $V_1^w = V_1$; if $T$ does not centralise $V_1$ then $w$ gives rise to a field automorphism on $V_1 \rtimes T$; hence $T$ centralises $V_1$, against $T$-minimality. Therefore $\rk V_1 = 1$. If $V_1^w = V_1$ consider $V_1$; if not, consider $V/(V_1+V_1^w)$. In any case $T$ which is inverted by $w$ acts on a rank $1$, $w$-invariant section, and therefore centralises it.

Hence $C_V(\alpha) \neq 0$, and Step \ref{p:Pr2=1:PSL2:st:unipotent} contradicts $U_\alpha \neq 1$.
\end{proofclaim}

The analysis of $V$ cannot be pushed beyond a certain limit. Of course if $V_1 = C_V\o(U_\alpha)$ is TI we find a contradiction; but if it is not, one can imagine having inside $G$ a bad unipotent centraliser: see the comment after the proof of Proposition \ref{p:Pr2>0:mixedBorel}. So we need to inspect the inner structure of $G$ more closely; this will be done in the quotient $G/\<\alpha^2\>$ (recall from Step \ref{p:Pr2=1:PSL2:st:alpha} that $\alpha^2 \in Z(G)$).

\begin{step}
Contradiction.
\end{step}
\begin{proofclaim}
We sum up the information: $\rk U_\alpha = \rk T = 1$ and the Sylow $2$-subgroup is connected.
We move to $\overline{G} = G/\<\alpha^2\>$ where this holds as well and $\overline{\alpha}$ is an involution. By connectedness of the Sylow $2$-subgroup, strongly real elements are unipotent; their set is non-generic (for instance \cite[Theorem 1]{BCSemisimple}).
Let us consider the definable function which maps two involutions of $\overline{G}$ to their product.

Let $\overline{r} = \overline{\alpha}\cdot\overline{\beta}$ be a generic product of conjugates of $\overline{\alpha}$. Then $\overline{C} = C\o_{\overline{G}}(\overline{r})$ is soluble, since otherwise the preimage $(\pi^{-1}(\overline{C}))\o = C_G\o(r)$ is non-soluble, whence irreducible by Proposition \ref{p:Pr2=1:LORO}: induction applied to $C_G\o(r)$ yields a contradiction. If $\overline{C}$ is a good torus, then by connectedness of $S$ and $\overline{S}$, one finds $\overline{\alpha} \in \overline{C}$, a contradiction. So $\overline{C}$ contains a non-trivial unipotent subgroup. Let $\overline{B}$ be the only Borel subgroup of $\overline{G}$ containing $\overline{C}$ (uniqueness follows from uniqueness principles); $\overline{\alpha}$ normalises $\overline{B}$. $\overline{B}$ is not unipotent, as it would generically cover $\overline{G}$ by uniqueness principles, which is against \cite[Theorem 1]{BCSemisimple} again. So $\overline{B}$ contains a conjugate of $\overline{T}$ which we may, by a Frattini argument, assume to be $\overline{\alpha}$-invariant. Still by connectedness of $\overline{S}$, one has $\overline{\alpha} \in \overline{B}$. Hence $\overline{\alpha}$ is an involution of $\overline{B}$; such elements are conjugate over $\overline{U} = U_p(\overline{B})$.

It is then clear that the fibre over the generic strongly real element $\overline{r}$ has rank $\leq m = \rk \overline{U}$.
Since $C_{\overline{G}}\o(\overline{\alpha}) = \overline{B_\alpha}$, one gets the estimate:
\[2 (\rk G - 2) - m < \rk G\]
that is $\rk G \leq m + 3 \leq 5$ by Step \ref{p:Pr2=1:PSL2:st:unipotent}. But $\rk G \neq 4$ by Wiscons' analysis and Proposition \ref{p:reductions}, so $\rk \overline{U} = 2$.

Here is the contradiction concluding the analysis. We lift $\overline{B}$ to a Borel subgroup $B$ of $G$; $B$ has rank $3$. But we know that $V_1 = C_V\o(U_\alpha)$ has rank $1$ by Step \ref{p:Pr2=1:PSL2:st:unipotent}; moreover $B_\alpha = N_G\o(V_1)$ by Proposition \ref{p:Pr2=1:LORO}. If $V_1$ is TI then $\cork B_\alpha = 2$ and $\rk G = 4$: a contradiction. So $V_1$ is not and we find a bad group $K$ of rank $3$ containing $U_\alpha$. It must intersect $B$ non-trivially; so up to conjugacy in $K$, $U_\alpha \leq B$, against maximality of $U_\alpha$ as a unipotent subgroup.
\end{proofclaim}

This concludes the Pr\"{u}fer rank $1$ analysis.
\end{proof}

\subsection{The Pr\"{u}fer Rank 2 Analysis}\label{s:Pr2=2}

We now suppose $\Pr_2(G) = 2$ and shall show that $G \simeq \SL_3(\K)$ acts on $V$ as on its natural module. Unfortunately we cannot rely on Altseimer's unpublished work aiming at identification of $\PSL_3(\K)$ \cite[Theorem 4.3]{AContributions} through the structure of centralisers of involutions. There also exists work by Tent \cite{TSplit} but as it involves $BN$-pairs, it is farther from our methods.
Instead we shall construct a vector space structure on $V$ for which a large subgroup of $G$ will be linear.

\subsubsection*{More Material}

Technically speaking this section is quite different; the two main ingredients are strongly embedded subgroups, defined before Proposition \ref{p:Pr2=2:noPSL2}, and the Weyl group, defined as follows: $W = N_G(S\o)/C_G(S\o)$.
The Weyl group has been abundantly studied and defined in the past; this definition will suffice for our needs.


\subsubsection{Central Involutions}

\begin{proposition}\label{p:Pr2=2:centralinv}
Suppose that $\Pr_2(G) = 2$. If there is a central involution in $G$ then $S$ and $N_G(S\o)$ have degree at most $2$.
\end{proposition}
\begin{proof}
Suppose there is a central involution, say $k \in S\o$ by torality principles. Observe that $k$ inverts $V$.

Then the other two involutions in $S\o$ do not have the same signature in their actions on $V$: they may not be conjugate. It follows from torality principles that $G$ has exactly three conjugacy classes of involutions, and that all elements in $N_G(S\o)$ centralise the involutions in $S\o$. In particular, the Weyl group has exponent at most $2$ and order at most $4$ (see \cite[Consequence of Fact 1]{Dprank} if necessary). Hence $N_G(S\o) = C_G(S\o) \cdot S$.

The argument bounding the order will resemble the one in Step \ref{p:Pr2=1:PSL2:st:S=So} of Proposition \ref{p:Pr2=1:PSL2}.
Suppose the order of $W$ is $4$. Then by \cite{Dprank} again there is an element $w \in S$ inverting $S\o$.
Let $S_0 < S\o$ be a $2$-torus of Pr\"{u}fer $2$-rank $1$ containing $k$ and $\Sigma = d(S_0)$.
Let $V_1$ be a $\Sigma$-minimal subgroup of $V$. If $V_1 = V$ then $w$ gives rise to a finite-order field automorphism on $V_1 \rtimes \Sigma$: a contradiction. If $\rk V_1 = 2$ then $V_1 \cap V_1^w$ is infinite, so by $\Sigma$-minimality $V_1^w = V_1$; if $\Sigma$ does not centralise $V_1$ then $w$ gives rise to a field automorphism on $V_1 \rtimes \Sigma$; hence $\Sigma$ centralises $V_1$, against $\Sigma$-minimality. Therefore $\rk V_1 = 1$. If $V_1^w = V_1$ consider $V_1$; if not, consider $V/(V_1+V_1^w)$. In any case $\Sigma$ which is inverted by $w$ acts on a rank $1$, $w$-invariant section, and therefore centralises it. Hence $C_V(\Sigma) \neq 0$, a contradiction to $k$ inverting $V$.
\end{proof}

\subsubsection{Removing $\SL_2(\K) \times \K^\times$}

\begin{proposition}\label{p:Pr2=2:noSL2}
Suppose that $\Pr_2(G) = 2$. Then $G$ contains no definable copy of $\SL_2(\K) \times \K^\times$.
\end{proposition}
\begin{proof}
The proof will closely follow that of Proposition \ref{p:Pr2=1:LORO}. There are a few differences and we prefer to replicate parts of the previous argument instead of giving one early general statement in the Pr\"{u}fer rank $1$ analysis.

\begin{step}\label{p:Pr2=2:noSL2:st:H}
Any definable, connected, reducible, non-soluble subgroup $H \leq G$ with $\Pr_2(H) \leq 1$ has the form $U \rtimes C$, where $C \simeq \SL_2(\L)$ and the central involution $i \in C$ inverts the $p$-unipotent group $U$; $\rk H \neq 4, 6$. Moreover if $H$ has a rank $1$ submodule $V_1$ then $V_1 = V^{+_i} = C_V\o(H)$; if $H$ has a rank $2$ submodule $V_2$ then $U$ centralises $V_2 = V^{-_i}$.
\end{step}
\begin{proofclaim}
This is exactly the proof of Step \ref{p:Pr2=1:LORO:st:H} of Proposition \ref{p:Pr2=1:LORO} (notice the extra assumption).
\end{proofclaim}

We start a contradiction proof: suppose that $G$ contains a subgroup isomorphic to $\SL_2(\K) \times \K^\times$.

\begin{step}\label{p:Pr2=2:noSL2:st:Sylow}
Sylow $2$-subgroups of $\SL_2(\K) \times \K^\times$ are Sylow $2$-subgroups of $G$. In particular, $G$ has three conjugacy classes of involutions; $\rk T = 2$ and $C_V(S\o) = 0$.
\end{step}
\begin{proofclaim}
We first find an involution central in $G$. Set $K = \SL_2(\K) \times \K^\times$. Let $i$ be the involution in $K' \simeq \SL_2(\K)$; let $j$ be the involution in $Z\o(K) \simeq \K^\times$ and $k = ij$. By the rank $3k$ analysis we know that $K'\simeq \SL_2(\K)$ acts naturally on $V_2 = V^{-_i} \simeq \K^2$ and centralises $V_1 = V^{+_i}$. Now observe that by irreducibility of $K'$ on $V_2$, $j$ either centralises or inverts $V_2$. If $j$ centralises $V_2$ and is not central then it inverts $V_1$: and $k = ij$ inverts $V_2 + V_1 = V$. If $j$ inverts $V_2$ and is not central then it centralises $V_1$: and $k = ij$ centralises $V_2 + V_1 = V$, a contradiction. In either case there is a central involution.

By Proposition \ref{p:Pr2=2:centralinv} the Sylow $2$-subgroup of $G$ is as described. Moreover $C_V(S\o) = 0$ since the central involution inverts $V$. Finally $V^{-_i}$ is not $T$-minimal: if it is, fix some torus $\Theta$ of $K'$; since $\Theta$ acts non-trivially, $V^{-_i}$ is $\Theta$-minimal as well: a contradiction. So $\ell_T(V) = 3$ and this shows $\rk T = 2$.
\end{proofclaim}

\begin{notationinproof*}
Let $C \leq G$ be isomorphic to $\SL_2(\K)$ and $i \in C$ be the central involution.
\end{notationinproof*}

\begin{notationinproof*}
Let $v_+ \in V^{+_i}\setminus\{0\}$ and $v_- \in V^{-_i} \setminus\{0\}$. Set $H_+ = C_G\o(v_+)$ and $H_- = C_G\o(v_-)$.
\end{notationinproof*}

\begin{step}[cf. Step \ref{p:Pr2=1:LORO:st:H+H-} of Proposition \ref{p:Pr2=1:LORO}]
Both $H_+$ and $H_-$ have corank $2$ or $3$ but not both have corank $3$. Moreover $H_+ \simeq U_+ \rtimes C$ where $U_+$ is a $p$-unipotent group inverted by $i$ and $\rk H_+ \neq 4, 6$; whereas $H_- = U_- \rtimes \Theta$ where $U_-$ is a $p$-unipotent group and $\Theta$ is a good torus of rank at most $1$.
\end{step}
\begin{proofclaim}
Since $C_V(S\o) = 0$ by Step \ref{p:Pr2=2:noSL2:st:Sylow}, any centraliser $C_G(v)$ with $v \in V\setminus\{0\}$ has Pr\"{u}fer rank at most $1$. This deals with $H_+$ and we turn to $H_-$.

We claim that $H_-$ has a connected Sylow $2$-subgroup. Suppose not: say $\tau \cdot \<w\> \leq H_-$ is a $2$-subgroup with $w \notin \tau \simeq \Z_{2^\infty}$. Then by connectedness of $H_-$ and torality principles, $w$ inverts $\tau = [\tau, w]$. Conjugating in $G$ into $S$, the involution $j \in \tau$ is a $G$-conjugate of $i$. But with a Frattini argument we may assume that $i$ normalises $\tau \cdot \<w\>$, so $[i, j] = 1$. By the structure of the Sylow $2$-subgroup of $G$, we find $i \in H_-$: a contradiction.

It follows that $v_+$ and $v_-$ are not $G$-conjugate. Also, connectedness of the Sylow $2$-subgroup of $H_-$ easily proves solubility: otherwise use induction on irreducible subgroups on the one hand and the structure of reducible subgroups (Step \ref{p:Pr2=2:noSL2:st:H}) on the other hand to find a contradiction. Finally, since $\rk T = 2$, good tori in $H_-$ have rank at most $1$.
\end{proofclaim}

\begin{step}
$\rk G \leq 6$.
\end{step}
\begin{proofclaim}
Let $x, y \in G$ be independent generic elements.

If $U_-$ centralises a rank $2$ module $V_2 \leq V$ then $(U_-\cap U_-^x)\o$ centralises $V_2 + V_2^x = V$, so $(U_-\cap U_-^x)\o = 1$. In that case $H_-$ can intersect at most over a toral subgroup, which has rank at most $1$: hence $\rk G \leq 2 \cork H_- + 1$. Notice that if we are not done then $\cork H_- = 3$, forcing $\cork H_+ = 2$.

If $U_-$ centralises a rank $1$ module $V_1 \leq V$ then $H_-$ normalises it; so $(H_-\cap H_-^x\cap H_-^y)\o$ contains no unipotence and is at most a toral subgroup of rank at most $1$; now $\rk G \leq 3 \cork H_- + 1$. If we are not done, then either $\cork H_- = 3$, in which case $\cork H_+ = 2$, or $\cork H_- = 2$ and $\rk G = 7$. In the latter case, $\rk H_+ \neq 4, 6$ forces $\cork H_+ = 2$ again.

The end of the argument is exactly like in Step \ref{p:Pr2=1:LORO:st:H} of Proposition \ref{p:Pr2=1:LORO}.
%
%
\end{proofclaim}

\begin{step}
Contradiction.
\end{step}
\begin{proofclaim}
If $\rk G = 6$ then $\rk H_+ = 3$ and $\cork H_+ = 3$; $v_+^G$ is generic in $V$. The argument for Step \ref{p:Pr2=1:LORO:st:contradiction} of Proposition \ref{p:Pr2=1:LORO} cannot be used (we leave it to the reader to see why).
But $\rk H_- = 4$, so for generic $x \in G$, $(H_-\cap H_-^x)\o$ has rank at least $2$: it contains a non-trivial unipotent subgroup $Y$. If $U_-$ centralises a rank $2$ module $W_2$ then $Y$ centralises $W_2 + W_2^x = V$, a contradiction. Hence $W_1 = C_V\o(U_-)$ has rank exactly $1$.

Still assuming $\rk G = 6$, let us show that $W_1$ is TI. Otherwise let $w_1 \in W_1\cap W_1^g \setminus\{0\}$ for $g \notin N_G(W_1) \geq N_G(U_-)$. Then $L = C_G\o(w_1) \geq \<U_-, U_-^g\>$ has rank $4$ and $U_-$ has rank $3$; $L$ is clearly soluble. If $L > U_p(L)$ then $U_- = U_p(L) = U_-^g$, a contradiction. If $L = U_p(L)$ then $C_V\o(L) \neq 0$, showing $W_1 = C_V\o(L) = W_1^g$, a contradiction again.

But always under the assumption that $\rk G = 6$, $C_G\o(v_+) = H_+ \simeq \SL_2(\K)$ so $W_1^G$ may not intersect $v_+^G$. Therefore $W_1^G$ is not generic, showing that $N = N_G\o(W_1)$ has corank $1$. By Hrushovski's Theorem, $G$ has a (necessarily non-soluble by Proposition \ref{p:reductions}) normal subgroup of corank $1, 2$, or $3$ contained in $N$; because $G$ contains $H_+ \simeq \SL_2(\K)$ which does not normalise $W_1$, the corank is $3$. So $G$ has either a normal bad subgroup of rank $3$, or a normal copy of $\pSL_2(\L)$. Using $2$-tori of automorphisms, every case quickly leads to a contradiction.

Hence $\rk G \leq 5$, proving that $G$ is an $N_\circ^\circ$-group: against Proposition \ref{p:LORO:Pr2=1}.
\end{proofclaim}

There are therefore no definable copies of $\SL_2(\K) \times \K^\times$ inside $G$.
\end{proof}

\subsubsection{Strongly Embedded Methods 1: Removing $\PSL_2(\K) \times \K^\times$}

Before reading the next proposition, remember that the case of $\PSL_2(\K) \times \K^\times$ was dealt with in Proposition \ref{p:reductions}.

Also recall from \cite[\S I.10.3]{ABCSimple} that a strongly embedded subgroup of a group $G$ is a definable, proper subgroup $H < G$ containing an involution, but such that $H\cap H^g$ contains no involution for $g \notin H$. By \cite[Theorem 10.19]{BNGroups} it actually suffices to check that $H$ contains the normaliser of a Sylow $2$-subgroup $S$ of $G$, and that for any involution $i \in S$ one has $C_G(i) \leq H$. Moreover if $H < G$ is strongly embedded in $G$ then $G$ conjugates its involutions.

\begin{proposition}\label{p:Pr2=2:noPSL2}
Suppose that $\Pr_2(G) = 2$. Then $G$ contains no definable copy of $\PSL_2(\K) \times \K^\times$.
\end{proposition}
\begin{proof}

\begin{notationinproof*}
Let $H \leq G$ be isomorphic to $\PSL_2(\K) \times \K^\times$.
\end{notationinproof*}

If $H = G$ then we contradict Proposition \ref{p:reductions}: hence $H$ is proper.
So $H < G$; any extension of $H$ is irreducible; since we are after a contradiction, we may suppose $G$ to be a minimal counter-example: $H$ is then a definable, connected, proper, maximal subgroup.
We shall prove that $H$ is strongly embedded in $G$, which will be close to the contradiction.

\begin{notationinproof*}
Let $\hat{\Theta} = \Theta \rtimes \<w\>$ be a Sylow $2$-subgroup of $H' \simeq \PSL_2(\K)$ and $i$ be the involution in $\Theta$; we may assume $\Theta \leq T$.
\end{notationinproof*}

Since the action of $H'$ on $V$ is known to be the adjoint action by the rank $3k$ analysis, we note that $V^{+_i} = C_V(\Theta) \leq V^{-_w}$. Besides $\ell_T(V) = 3$ for the same reason as in Step \ref{p:Pr2=2:noSL2:st:Sylow} of Proposition \ref{p:Pr2=2:noSL2}, so $\rk T = 2$ and $T \leq H$.
Moreover, since the action of $H'$ is irreducible, the involution in $Z(H) \simeq \K^\times$ inverts $V$ and is central in $G$.
As a consequence of Proposition \ref{p:Pr2=2:centralinv}, a Sylow $2$-subgroup of $H$ is one of $G$ as well. But no subquotient of the Sylow $2$-subgroup of $H$ is isomorphic to the Sylow $2$-subgroup of $\SL_2(\L)$; as a consequence, $G$ has no subquotient isomorphic to $\SL_2(\L)$.

\begin{step}\label{p:Pr2=2:noPSL2:st:Ci}
$C_G\o(i) = T \leq H$ (and likewise for $w$ and $iw$ with another torus).
\end{step}
\begin{proofclaim}
By $H'$-conjugacy it suffices to deal with $i$. First suppose that $C_G\o(i)$ is non-soluble. By reducibility, $C_G\o(i)$ has a subquotient isomorphic to $\SL_2(\L)$: against our observations on the Sylow $2$-subgroup.
Hence $C_G\o(i)$ is soluble, say $C_G\o(i) = U \rtimes T$. Now $U$ normalises both $V^{+_i}$ (which has rank $1$) and $V^{-_i}$, so $\rk C_V\o(U) \geq 2$ and $V^{+_i} \leq C_V\o(U)$. But $w$ centralises $i$ so it normalises $U$: hence $w$ normalises $V/C_V\o(U)$. Since $w$ inverts $\Theta \leq T$ and there are no field automorphisms in our setting, $\Theta$ centralises $V/C_V\o(U)$. This shows $V \leq C_V\o(U) + C_V(\Theta) = C_V\o(U) + V^{+_i} = C_V\o(U)$, and therefore $U = 1$.
\end{proofclaim}

\begin{notationinproof*}
Let $\alpha \in Z(H) \simeq \K^\times$ have minimal order with $\alpha \notin Z(G)$.
\end{notationinproof*}

This certainly exists as $Z(G)$ is finite by Proposition \ref{p:reductions}.
By maximality of $H$, $C_G\o(\alpha) = H$ and $C_G\o(\alpha^2) = G$; moreover $(i\alpha)^2 \neq 1$.

\begin{step}\label{p:Pr2=2:noPSL2:st:Cialpha}
$C_G\o(i\alpha) = T$ (and likewise for $w\alpha$ and $iw\alpha$ with another torus).
\end{step}
\begin{proofclaim}
By $H'$-conjugacy it suffices to deal with $i\alpha$. If $C_G\o(i\alpha)$ is non-soluble, then by induction it must be reducible, and $G$ has a subquotient isomorphic to $\SL_2(\L)$: a contradiction.
Hence $C_G\o(i\alpha)$ is soluble, say $C_G\o(i\alpha) = U \rtimes T$. Now $i$ normalises $U$, so by Step \ref{p:Pr2=2:noPSL2:st:Ci}, $i$ inverts $U$. But so do $w$ and $iw$: therefore $U = 1$.
\end{proofclaim}

\begin{step}
Contradiction.
\end{step}
\begin{proofclaim}
Let $\overline{G} = G/\<\alpha^2\>$ and denote the image of $g \in G$ by $\overline{g}$.
First, by Proposition \ref{p:Pr2=2:centralinv} and the connectedness of centralisers of decent tori \cite{ABAnalogies}, $N_G(S) \leq N_G(S\o) = C_G(S\o) \cdot S \subseteq C_G\o(i) \cdot S \subseteq H$, which goes to quotient modulo $\<\alpha^2\>$ so that $N_{\overline{G}}(\overline{S}) \leq N_{\overline{G}}(\overline{S}\o) \leq \overline{H}$.

By Steps \ref{p:Pr2=2:noPSL2:st:Ci} and \ref{p:Pr2=2:noPSL2:st:Cialpha}, for any involution $\overline{\ell} \neq \overline{\alpha}$ in $\overline{S}$, one has $C_{\overline{G}}\o(\overline{\ell}) = \overline{T} \leq \overline{H}$; by construction, $C_{\overline{G}}\o(\overline{\alpha}) = \overline{H}$. Be careful that checking connected components does not suffice for strong embedding.

But by torality principles, $\overline{\ell}$ is $\overline{H}$-conjugate to an involution in $\overline{S}\o$, so we may assume $\overline{\ell} \in \overline{S}\o$; then by a Frattini argument, $C_{\overline{G}}(\overline{\ell}) \subseteq C_{\overline{G}}\o(\overline{\ell})\cdot N_{\overline{G}}(\overline{S}\o)$; now $N_{\overline{G}}(\overline{S}\o) = C_{\overline{G}}(\overline{S}\o)\cdot \overline{S}$ by Proposition \ref{p:Pr2=2:centralinv} again, so using the connectedness of centralisers of decent tori one more time:
\[C_{\overline{G}}(\overline{S}\o) \leq C_{\overline{G}}\o(\overline{i}) = \overline{T} \leq \overline{H}\]
This shows $C_{\overline{G}}(\overline{\ell}) \leq \overline{H}$ and the whole paragraph also applies to $\overline{\ell} = \overline{\alpha}$.

Hence $\overline{H}$ is strongly embedded all right and $\overline{G}$ conjugates its involutions. This induces an element of order $3$ in the Weyl group of $\overline{G}$ and of $G$ as well: a contradiction.
\end{proofclaim}

There are therefore no definable copies of $\PSL_2(\K) \times \K^\times$ inside $G$.
\end{proof}

\subsubsection{Strongly Embedded Methods 2: Classical Involutions}

\begin{proposition}\label{p:Pr2=2:classical}
If $\Pr_2(G) = 2$ then all involutions in $G$ satisfy $C_G\o(\ell) \simeq \GL_2(\K)$.
\end{proposition}
\begin{proof}

Call an involution $i \in G$ \emph{meek} if $C_G\o(i)$ is soluble.

\begin{step}\label{p:Pr2=2:SL3:st:centralisers}
If an involution $i \in G$ is neither meek nor central, then $C_G\o(i) \simeq \GL_2(\K)$.
\end{step}
\begin{proofclaim}
Let $C = C_G\o(i)$. Since $i$ is not central in $G$, it does not invert $V$: we get a decomposition $V = V_1 \oplus V_2$ where $\rk V_r = r$, and both are $C$-invariant. Set $D = C_C(V_2)$. Now $D$ is faithful on $V_1$, so it is abelian-by-finite. By assumption $C$ is not soluble, so by connectedness $C/D$ is not either. By the rank $2$ analysis, $C/D \simeq \SL_2(\K)$ or $C/D\simeq \GL_2(\K)$ in their natural actions on $V_2\simeq \K^2$.

First suppose that $C/D \simeq \SL_2(\K)$. Then $(C/D\o)/(D/D\o) \simeq C/D \simeq \SL_2(\K)$ so by \cite[Theorem 1]{ACCentral}, $D = D\o$. Notice that $D\o$ contains a $2$-torus of rank $1$; by Zilber's Field Theorem, $D \simeq \L^\times$ for some field structure $\L$ of rank $1$ in the action on $V_1 \simeq \L_+$. Let $E = C_C\o(V_1)$; since $\cork_C(E) = 1 = \rk D$, one finds $C = E \times D$, against Proposition \ref{p:Pr2=2:noSL2}.

Now suppose that $C/D \simeq \GL_2(\K)$. Then $D\o$ has no involutions, so it centralises $V_1$: $D$ is therefore finite. Since $\SL_2(\K) \simeq (C/D)' = C'D/D \simeq C'/(C'\cap D)$, \cite[Theorem 1]{ACCentral} again forces $C'\simeq \SL_2(\K)$. Moreover:
\[\K^\times \simeq \GL_2(\K)/\SL_2(\K) \simeq (C/D)/(C'D/D) \simeq C/C'D \simeq (C/C')/(C'D/C')\]
so a finite quotient of, and therefore $C/C'$ itself, is definably isomorphic to $\K^\times$.
Finally let $\Theta \leq C$ be a maximal good torus: $C = C'\cdot \Theta = C' \ast C_{\Theta}(C') = C' \ast Z\o(C)$ where the intersection is a subgroup of $Z(C') \simeq \Z/2\Z$.
By Proposition \ref{p:Pr2=2:noSL2}, the intersection is not trivial, so that $C \simeq \GL_2(\K)$.
\end{proofclaim}

\begin{step}\label{p:Pr2=2:SL3:st:nocentral}
There is no central involution.
\end{step}
\begin{proofclaim}
Suppose there is a central involution, say $k$, and let $i, j$ be the other two involutions in $S\o$. Of course $C_G\o(i) = C_G\o(j)$.
If $i$ and $j$ are not meek then by Step \ref{p:Pr2=2:SL3:st:centralisers}, $C_G\o(i) = C_G\o(j) \simeq \GL_2(\K)$, which has only one central involution, a contradiction.
Hence $i$ and $j$ are both meek.

As a consequence, $G$ has no definable subgroup isomorphic to $\SL_2(\K)$: for if $H$ is one such then the central involution in $H$ cannot be meek, so it is $k$; but $k$ inverts $V$, against the rank $3k$ analysis.

We claim that $G$ actually has no definable subquotient isomorphic to $\SL_2(\K)$. Suppose $H/K \simeq \SL_2(\K)$ is one. If $K$ has no involutions, then like in Step \ref{p:Pr2=2:noSL2:st:H} of Proposition \ref{p:Pr2=2:noSL2}, we may lift $H/K$ to a genuine copy of $\SL_2(\K)$ inside $H$: a contradiction. So $K$ does have involutions; as we argued a number of times, $K$ is connected and soluble, so we find $K = U \rtimes \Theta$ with $\Theta$ a good torus of Pr\"{u}fer $2$-rank $1$.
Now by the conjugacy of good tori in $K$, $H = N_H(\Theta)\cdot U$ and $N_H(\Theta)/N_K(\Theta) \simeq H/K$, so we may assume $\Theta$ to be normal, and therefore central, in $H$. The involution in $\Theta$ must then be $k$.
If there is a rank $1$, $H$-minimal module $V_1 \leq V$, then $C_H(V_1) < H$ has corank $1$; we find $H = C_H(V_1) \cdot \Theta$ and $C_H\o(V_1)/C_K(V_1) \simeq H/K \simeq \SL_2$, but $C_K(V_1)$ now has no involutions: we are done. If there is a rank $2$, $H$-minimal module $V_2 \leq V$ then we argue similarly with $C_H(V/V_2)$. If $V$ is $H$-minimal then we use induction and find $H \simeq \PSL_2(\L) \times \L^\times$, which is against having a subquotient isomorphic to $\SL_2(\K)$.

It follows from Proposition \ref{p:Pr2=2:noPSL2} that $G$ is an $N_\circ^\circ$-group, and Proposition \ref{p:LORO:Pr2=1} forces $\Pr_2(G) = 1$, a contradiction.
\end{proofclaim}

\begin{step}
$S\o$ contains (at least) one involution $k$ with $C_G\o(k) \simeq \GL_2(\K)$.
\end{step}
\begin{proofclaim}
This is a proper subset of the previous argument.
\end{proofclaim}

\begin{step}
$S\o$ contains (at least) two involutions $k \neq \ell$ with $C_G\o(k) \simeq \GL_2(\K)$ and $C_G\o(\ell) \simeq \GL_2(\L)$.
\end{step}
\begin{proofclaim}
If there is exactly one, say $k$, then the other two, say $i$ and $j$, are meek. We shall construct a strongly embedded subgroup.

Immediately notice that the Weyl group of $C_G\o(k) \simeq \GL_2(\K)$ gives rise to a $2$-element $w$ exchanging $i$ and $j$ but fixing $k$.
Let $C_G\o(i) = U_i \rtimes T$ and $C_G\o(j) = U_j \rtimes T$. Notice that $U_i$ normalises both $V^{+_i}$ and $V^{-_i}$, so $\rk C_V(U_i) \geq 2$.

First suppose that $C_V(U_i) \neq C_V(U_j)$. Then $k$ may not invert $C_V(U_i)$ since applying $w$, it would invert $C_V(U_j)$ as well and therefore invert all of $C_V(U_i) + C_V(U_j) = V$, against Step \ref{p:Pr2=2:SL3:st:nocentral}. Since $V^{+_k}$ has rank $1$, we find $V^{+_k} \leq (C_V(U_i)\cap C_V(U_j))\o$ and equality follows. So $H = N_G(V^{+_k})$ contains $\<C_G\o(k), C_G\o(i), C_G\o(j)\>$.

Now suppose that $C_V(U_i) = C_V(U_j)$. Observe from $C_G\o(k) \simeq \GL_2(\K)$ that $\Theta = [T, w]$ contains $k$. Now $w$ inverts $\Theta$ which normalises $V/C_V(U_i)$, and therefore $\Theta$ centralises $V/C_V(U_i)$. Hence $V \leq C_V(U_i) + V^{+_k}$. If $U_i \neq 1$, then $k$ inverts $C_V(U_i)$, showing $C_V\o(U_i) = V^{-_k}$. In that case, $H = N_G(V^{-_k})$ contains $\<C_G\o(k), C_G\o(i), C_G\o(j)\>$; notice that this is also true if $U_i = 1$.

We claim that $H$ is strongly embedded in $G$.

Let us first show that $C_G(k)$ is connected. Let $c \in C_G(k)$; lifting torsion, we may suppose $c$ to have finite order (as a matter of fact, by Steinberg's Torsion Theorem \cite{DSteinberg} $c$ may be taken to be a $2$-element). Then $c$ induces an automorphism of $H_k = (C_G\o(k))' \simeq \SL_2(\K)$, so by \cite[Theorem 8.4]{BNGroups}, $c \in H_k \cdot C_G(H_k)$. Now fix any algebraic torus $\Theta$ of $H_k$: by connectedness of centralisers of tori \cite{ABAnalogies}, $C_G(H_k) \leq C_G(\Theta) = C_G\o(\Theta) \leq C_G\o(k)$. This shows $C_G(k) = C_G\o(k)$.
As a consequence, $N_G(S) \leq C_G(k) \leq H$.

Now let $\ell \in S$: we show $C_G(\ell) \leq H$. Notice that $\ell \in C_G(k) = C_G\o(k) \simeq \GL_2(\K)$, so conjugating in $C_G\o(k)$ we may suppose $\ell = i$ or $\ell = k$. The latter case is known since $C_G(k)$ is connected. So we may suppose $\ell = i$. But if $c \in C_G(i)\setminus C_G\o(i)$, lifting torsion and using Steinberg's Torsion Theorem we may suppose $c$ to be a $2$-element. By a Frattini argument, $c$ normalises some maximal $2$-torus $\Sigma\o \leq C_G\o(i)$. Let $\kappa$ be the non-meek involution in $\Sigma\o$; since $S\o \leq C_G\o(i)$, $\kappa$ and $k$ are conjugate in $C_G\o(i)\leq H$, say $\kappa = k^h$. Now $c$ centralises $\kappa$ so $c \in C_G(\kappa) = C_G(k)^h \leq H$: we are done.

Since $G$ has a strongly embedded subgroup, it conjugates its involutions: so $i$ is conjugate to $k$, against meekness.
\end{proofclaim}

Finally let $i, j \in S\o$ have centralisers$\o$ isomorphic to $\GL_2(\K)$ and $\GL_2(\L)$. Then $C_G\o(i)$ and $C_G\o(j)$ give rise to two distinct transpositions on the set of involutions of $S\o$, meaning that the Weyl group is transitive on the set of involutions of $S\o$. As a consequence, $i, j$, and $k = ij$ are conjugate.
\end{proof}

\subsubsection{Der Nibelungen Ende}

\begin{proposition}
If $\Pr_2(G) = 2$ then $G\simeq \SL_3(\K)$ in its natural action on $V \simeq \K^3$.
\end{proposition}
\begin{proof}
As before, let $i, j, k$ be the involutions in $S\o$.

\begin{step}
There are a $\K$-vector space structure on $V$ and an irreducible subgroup $H \leq G$ which is $\K$-linear.
\end{step}
\begin{proofclaim}
Let $H_i = (C_G\o(i))' \simeq \SL_2(\K)$; define $H_j$ and $H_k$ similarly. We know how $H_i$ acts on $V$: it centralises $V^{+_i}$ and acts on $V^{-_i} = V^{+_j} \oplus V^{+_k}$ as on its natural module, meaning that there is a (partial) $\K$-vector space structure on $V^{-_i}$.
We extend it to a global vector space structure on all of $V$ as follows.

First let $w \in C_G\o(k) \simeq \GL_2(\K)$ be an element of order $4$ exchanging $i$ and $j$ while fixing $k$, and notice that we may actually take $w \in (C_G\o(k))' = H_k \simeq \SL_2(\K)$; then $w$ centralises $V^{+_k}$ and $w^2 = k$.

Let $a_i \in V^{+_i}$. Then $a_i^{w^{-1}} \in C_V(i^{w^{-1}}) = V^{+_j}$, a $\K$-vector subspace of $V^{-_i}$, so it makes sense to define:
\[\lambda \cdot a_i := \left(\lambda \cdot a_i^{w^{-1}}\right)^w\]
This clearly maps $V^{+_i}$ into itself; moreover it is additive in $a_i$ and additive and multiplicative in $\lambda$. So we have extended the vector space structure to all of $V$.

We show that $H = \<H_i, w\>$ is linear. Since $H_i$ centralises $V^{+_i}$, it clearly is linear on $V = V^{-_i} + V^{+_i}$. For $w$ we argue piecewise. Let $a_i \in V^{+_i}$ and $a_j = a_i^{w^{-1}} \in V^{+_j}$. Then, bearing in mind that $w^2 = k$ inverts $V^{-_k} = V^{+_i} + V^{+_j}$:
\[\lambda \cdot a_i^w = \lambda \cdot a_j^{w^2} = \lambda \cdot (-a_j) = - \lambda \cdot a_j = (\lambda\cdot a_j)^{w^2} = (\lambda \cdot a_i)^w\]
Let $a_j \in V^{+_j}$ and $a_i = a_j^w \in V^{+_i}$. Now:
\[\lambda \cdot a_j^w = \lambda \cdot a_i = \left(\lambda\cdot a_i^{w^{-1}}\right)^w = (\lambda \cdot a_j)^w\]
Finally, let $a_k \in V^{+_k} = C_V(H_k) \leq C_V(w)$. Then $(\lambda \cdot a_k)^w = \lambda \cdot a_k = \lambda \cdot a_k^w$. The element $w$ is linear.
\end{proofclaim}

Notice that if $H = G$ then we are done, since although we did not bother to identify $H$ explicitly, $G$ is then a linear group. Now semi-simple, linear groups in characteristic $p$ are known to be algebraic \cite[Theorem 2.6]{MStructure} (which was already used in Step \ref{p:reductions:st:RG} of Proposition \ref{p:reductions}), and we conclude by inspection.

So suppose $H < G$: we shall find a contradiction.

\begin{step}
Contradiction.
\end{step}
\begin{proofclaim}
By induction, and in view of Proposition \ref{p:Pr2=2:noPSL2}, $H \simeq \SL_3(\K)$. Up to changing the vector space structure (which should however not be necessary), $V$ is the natural $H$-module.

Fix $v \in V\setminus\{0\}$ and let $K = C_G\o(v)$. First observe that $\Pr_2 (K) \leq 1$ since $C_V(S\o) = 0$ as observed from the action of $H \simeq \SL_3(\K)$.
Moreover $K \geq C_H\o(v)$ contains a copy of $\SL_2(\K)$ as seen by inspection. If $K$ is irreducible on $V$ then by induction $K \simeq \PSL_2(\L)$, a contradiction. So $K$ is reducible; by Step \ref{p:Pr2=2:noSL2:st:H} of Proposition \ref{p:Pr2=2:noSL2} again, write $K = U \rtimes C$ with $U$ a unipotent group and $C \simeq \SL_2(\K)$; moreover $\rk K \neq 4, 6$.

First suppose that $K$ has a rank $2$ module $V_2 \leq V$. Then we know that $U$ centralises $V_2$. Let $g \in G$ be generic and $v_2 \in V_2\cap V_2^g\setminus\{0\}$. Since $H$ is transitive on $V\setminus\{0\}$, $C_G\o(v_2) = Y \rtimes D$ for some conjugates $Y, D$ of $U, C$ respectively. Yet $Y\rtimes D \geq \<U, U^g\>$, and $U\cap U^g = 1$ since the intersection centralises $V_2 + V_2^g = V$. This proves $2 \rk U \leq \rk U + \rk C$, and $\rk U \leq 3$; since $\rk K \neq 6$, one finds $\rk G \leq K + 3 \leq 8$, against $H \simeq \SL_3(\K)$ being proper.

Now suppose that $K$ has a rank $1$ module $V_1$. Then we know that $K$ centralises $V_1$; in particular, for independent and generic $x, y \in G$, the intersection $(K\cap K^x \cap K^y)\o$ is trivial: it follows $\rk G \leq 3 \cork K \leq 9$. So $H < G$ has corank $1$; by Hrushovski's Theorem $G$ has a normal subgroup of rank $\geq 6$, which certainly intersects the quasi-simple group $H \simeq \SL_3(\K)$. Hence $H$ itself is normal in $G$; now $G = H \cdot C_G\o(H)$ by \cite[Theorem 8.4]{BNGroups}, and $C_G\o(H)$ is a normal subgroup of rank $1$, contradicting Proposition \ref{p:reductions}.
\end{proofclaim}

This concludes the Pr\"{u}fer rank $2$ analysis.
\end{proof}

\subsection{The Pr\"{u}fer Rank 3 Analysis}

This is a one-liner: \cite[Theorem 1.4]{BBGroups} settles the question.
On the other hand a direct proof along the lines of the Pr\"{u}fer rank $2$ argument would certainly be possible.
In any case our Theorem is proved.\qed

\bibliographystyle{alpha}
\bibliography{RMf}

\end{document}